\documentclass[12pt]{amsart}

\usepackage{amsmath}
\usepackage{amssymb}
\usepackage{amscd}
\usepackage{latexsym}
\usepackage{amsthm}
\usepackage{mathrsfs}
\usepackage{color}
\usepackage{amsfonts}
\usepackage{enumitem}
\usepackage{tikz}
\usepackage{tikz-cd}
\usepackage{mathabx,epsfig}
\usetikzlibrary{arrows, matrix}
\newtheorem{thm}{Theorem}[section]
\newtheorem*{thm*}{Theorem}
\newtheorem{cor}[thm]{Corollary}%[section]
\newtheorem*{cor*}{Corollary}
\newtheorem{lem}[thm]{Lemma}%[section]
\newtheorem*{lem*}{Lemma}
\newtheorem{prop}[thm]{Proposition}%[section]
\newtheorem*{prop*}{Proposition}
\theoremstyle{definition}
\newtheorem{defn}[thm]{Definition}%[section]
\newtheorem*{defn*}{Definition}
\newtheorem*{conjecture}{Conjecture}

\theoremstyle{remark}
\newtheorem{rem}[thm]{Remark}%[section]
\newtheorem*{rem*}{Remark}
\newtheorem{example}[thm]{Example}%[section]

\newtheorem*{problem*}{Problem}

\numberwithin{equation}{section}

\newcommand{\Q}{\mathbb Q}

\newcommand{\C}{\mathbb C}

\newcommand{\Z}{\mathbb Z}

\newcommand{\cA}{\mathcal A}
\renewcommand{\AA}{\mathbb A}

\newcommand{\DD}{\mathbb D}

\newcommand{\eps}{\epsilon}

\newcommand{\Ht}{\tilde{H}}

\DeclareMathOperator{\LLT}{LLT}
\DeclareMathOperator{\Id}{Id}
\DeclareMathOperator{\touch}{touch}

\DeclareMathOperator{\Des}{Des}
\DeclareMathOperator{\Dinv}{Dinv}
\DeclareMathOperator{\dinv}{dinv}

\DeclareMathOperator{\pExp}{Exp}

\DeclareMathOperator{\Sym}{Sym}
\DeclareMathOperator{\inv}{inv}
\DeclareMathOperator{\Inv}{Inv}
\DeclareMathOperator{\Std}{Std}
\DeclareMathOperator{\area}{area}
\DeclareMathOperator{\Area}{Area}

\DeclareMathOperator{\bounce}{bounce}

\renewcommand{\Im}{\mathrm{Im}}

\def\acts{\mathrel{\reflectbox{$\righttoleftarrow$}}}

\begin{document}

\title{A proof of the shuffle conjecture}
\author{Erik Carlsson}
\address{International Centre for Theoretical Physics \\
Str. Costiera, 11, 34151 Trieste, Italy}
\curraddr{Department of Mathematics, University of California, Davis \\
1 Shields Ave., Davis, CA 95616}
\email{ecarlsson@math.ucdavis.edu}
\author{Anton Mellit}
\address{International Centre for Theoretical Physics \\
Str. Costiera, 11, 34151 Trieste, Italy}
\curraddr{Faculty of Mathematics, University of Vienna \\
Oskar-Morgenstern-Platz 1, 1090 Vienna, Austria}
\email{anton.mellit@univie.ac.at}

%\date{}

\begin{abstract}
We present a proof of the compositional shuffle conjecture
\cite{haglund2012compositional},
which generalizes the famous shuffle conjecture
for the character of the diagonal coinvariant algebra
\cite{haglund2005diagcoinv}.
We first formulate the 
combinatorial side of the conjecture in terms
of certain operators on a graded vector space $V_*$
whose degree zero part is the ring of symmetric
functions $\Sym[X]$ over $\mathbb{Q}(q,t)$.
We then extend these operators to 
an action of an algebra $\tilde{\AA}$ acting on this space,
%larger algebras $\mathbb{A}$ and
%$\mathbb{A}^*$ acting on this
%space, 
and interpret the right generalization
of the $\nabla$ using an involution of the algebra
%operator as an intertwiner between
%the two actions, 
which is antilinear with
respect to the conjugation $(q,t)\mapsto (q^{-1},t^{-1})$.
\end{abstract}

\maketitle

\section{Introduction}

The shuffle conjecture 
of Haglund, Haiman, Loehr, Remmel, and Ulyanov
\cite{haglund2005diagcoinv} predicts a combinatorial
formula for the Frobenius
character $\mathcal{F}_{R_n}(X;q,t)$ of the diagonal coinvariant
algebra $R_n$ in $n$ pairs of variables, which is a symmetric function
in infinitely many variables with coefficients in
 $\mathbb{Z}_{\geq 0}[q,t]$.
%Here the weights $q,t$ are weights associated to
%the bigrading on $R_n$, and the Frobenius character
%associates a symmetric function is associated to an irreducible
%representation by the Schur-Weyl duality map.
By a result of Haiman \cite{Hai02},
the Frobenius character is given explicitly by
\[\mathcal{F}_{R_n}(X;q,t)=(-1)^n \nabla e_n[X],\]
where up to a sign convention,
$\nabla$ is the operator which is diagonal
in the modified Macdonald basis defined in 
\cite{Bergeron99identitiesand}.
%and we have used a different sign convention which is
%more natural for this paper.
%symmetric functions which is diagonal in the Haiman-Macdonald
%basis \eqref{nabdef}.
The original shuffle conjecture states
\begin{equation}
\label{shuffconj}
(-1)^n \nabla e_n[X] = 
\sum_{\pi} \sum_{w \in \mathcal{WP}_\pi} t^{\area(\pi)} q^{\dinv(\pi,w)} x_w.
\end{equation}
Here $\pi$ is a Dyck path of length $n$, and $w$ is some
extra data called a ``word parking function''
depending on $\pi$. The functions $(\area,\dinv)$
are statistics associated to a Dyck path and a parking function, and $x_w$
is a monomial in the variables $x$.
They proved that this sum,
denoted $D_n(X;q,t)$, is symmetric in the 
$x$ variables and so does at least define a symmetric
function. 
They furthermore showed that it included many previous
conjectures and results about the $q,t$-Catalan numbers,
and other special cases 
\cite{garsia1996remarkable,
garsia2002catalan,haglund2003conjectured,
egge2003sch,haglund2004sch}.
Remarkably, $D_n(X;q,t)$ had not even been
proven to be symmetric in the $q,t$ variables until now,
even though the symmetry of $\mathcal{F}_{R_n}(X;q,t)$ is obvious.
The name ``shuffle conjecture'' has to do with the fact
that the coefficient of $m_\mu$ in equation \eqref{shuffconj} 
can be expressed in terms parking functions that are
``$\mu$-shuffles.''
See Conjecture 6.1 of 
Haglund's book \cite{haglund2008catalan} for a detailed explanation,
and Chapter 6 in general for a thorough introduction to this topic.

%where we have used a different sign convention for
%the $\nabla$ operator which
%is more natural for our paper.
%The usual definition of $\nabla$ is as the operator
%%that multiplies the Haiman-Macdonald polynomials b%y
%the $q,t$ weight in \eqref{nabdef}.

In \cite{haglund2012compositional}
Haglund, Morse, and Zabrocki conjectured
a refinement of the original conjecture which partitions
$D_n(X;q,t)$ by specifying the points where the Dyck
path touches the diagonal called the ``compositional
shuffle conjecture.''
The refined conjecture states
\begin{equation}
\label{compshuff}
\nabla \left(C_{\alpha}[X;q]\right)=
\sum_{\touch(\pi)=\alpha} \sum_{w \in \mathcal{WP}_\pi} t^{\area(\pi)}q^{\dinv(\pi,w)} x_w.
\end{equation}
Here $\alpha$ is a composition, i.e. a finite list of positive
integers specifying the gaps between the touch points of $\pi$.
The function $C_{\alpha}[X;q]$ is defined below as a 
composition of creation operators for Hall-Littlewood polynomials
in the variable $1/q$. They proved that
\[\sum_{|\alpha|=n} C_\alpha[X;q]=(-1)^n e_n[X],\]
implying that \eqref{compshuff} does indeed generalize
\eqref{shuffconj}.
The right hand side of \eqref{compshuff} will be denoted
by $D_{\alpha}(X;q,t)$.
A desirable approach to proving
\eqref{compshuff} would be to determine a recursive formula
for $D_{\alpha}(X;q,t)$, and interpret the result in terms of
some commutation relations for $\nabla$. Indeed, this
approach has been applied in some important special cases, see 
\cite{garsia2002catalan,GXZ2012,hicks2012sharpening}.
In \cite{GXZ2012}, for instance, the 
authors devise a recursive formula
(Proposition 3.12) to prove the Catalan case of 
the compositional conjecture, extending the results of \cite{garsia2002catalan}.
Unfortunately, no such recursion is known in the general
case, and so an even more refined function is needed.

In this paper, we will construct the desired refinement
as an element of a larger vector space $V_k$
of symmetric functions
over $\Q(q,t)$ with $k$ additional variables $y_i$ adjoined,
where $k$ is the length of the composition $\alpha$,
\[N_{\alpha} \in V_k=\Sym[X][y_1,...,y_k].\]
In our first result, (Theorem \ref{thm:recN}), 
we will explain how to recover $D_\alpha(X;q,t)$ from $N_\alpha$,
%In Theorem \ref{thm:recN},
%and prove that $N_\alpha$ satisfies
which is defined as an explicit recursion.
% that completely determines it.
%we will define such a refinement called
%$N_\alpha$ and prove that it extends the function
%$D_{\alpha}(X;q,t)$
In fact, while they live in different vector spaces,
the recursions for $N_\alpha$ are similar to 
the recursions for the Catalan case in \cite{GXZ2012}. 
We make this connection precise
in Proposition \ref{catrecprop} below, which explains how
the latter formulas follow as a special case.

We then define a pair of algebras $\mathbb{A}$
and $\mathbb{A}^*$
which are isomorphic by an antilinear isomorphism
with respect to the conjugation
$(q,t)\rightarrow (q^{-1},t^{-1})$, as well as an explicit action
of each
% In Theorem \ref{mainthm},
%We then define an explicit action of both algebras
%on a graded vector space
on the direct sum $V_*=\bigoplus_{k\geq 0} V_k$.
We will then prove that there is an 
antilinear involution $N$ on $V_*$
which intertwines the two actions (Theorem \ref{mainthm}),
and represents an involutive automorphism on 
a larger algebra $\AA,\AA^*\subset \tilde{\AA}$.
This turns out to be the essential fact that relates the $N_\alpha$
to $\nabla$. 

The compositional shuffle conjecture (Theorem \ref{shuffthm}),
then follows as a simple corollary from the following properties:
%Given a composition
%$\alpha$ of length $k$, the element
%\[N_{\alpha}:= N(y_\alpha),\quad 
%y_\alpha= y_1^{\alpha_1}\cdots y_k^{\alpha_k} \in V_k\]
%will turn out to be the desired refinement mentioned
%in the last paragraph.
%we will find a more general function
%denoted $N_{\alpha} \in V_k$, where $k$ is the length of
%$\alpha$, and 
%Specifically, the shuffle conjecture will follow from the
%following properties:
%We will show that $N_\alpha$ has the following properties:
\begin{enumerate}
\item There is a surjection coming from $\mathbb{A},\mathbb{A}^*$
\[d_-^{k} : V_k \rightarrow V_0=\Sym[X]\]
which maps a monomial $y_\alpha$ in the $y$ variables
to an element $B_{\alpha}[X;q]$ which is similar to
$C_\alpha[X;q]$, and maps $N_\alpha$ to $D_{\alpha}(X;q,t)$, up to a sign.
\item The involution $N$
commutes with $d_-$, and maps $y_\alpha$ to $N_\alpha$. %and is antilinear with respect
%to the conjugation $(q,t)\rightarrow (q^{-1},t^{-1})$. 
\item The restriction of $N$ to $V_0=\Sym[X]$ 
agrees with $\nabla$ composed with 
a conjugation map which essentially exchanges the
$B_{\alpha}[X;q]$ and $C_{\alpha}[X;q]$.
\end{enumerate}
It then becomes clear that these properties imply \eqref{compshuff}.

While the compositional shuffle conjecture 
is clearly our main application,
the shuffle conjecture has been further generalized in
several remarkable directions such as the
rational compositional shuffle conjecture, and relationships to
knot invariants, double affine Hecke algebras,
and the cohomology of the affine Springer fibers, see
\cite{bergeron2014rational,
gorsky2014torus,gorsky2015refined,negut2013shuffle,hikita2014affine,
shiff2011elliptic, shiff2013elliptic}.
We hope that future applications to 
these fascinating topics will be forthcoming.

\subsection{Acknowledgments}
The authors would like to thank Fran\c{c}ois Bergeron, Adriano Garsia,
Mark Haiman, Jim Haglund, Fernando Rodriguez-Villegas and Guoce Xin
for many valuable discussions on this and related topics.
The authors also acknowledge
the International Center for Theoretical Physics, Trieste, Italy,
at which most of the research for this paper was performed.
Erik Carlsson was also supported by the
Center for Mathematical Sciences and Applications at Harvard University
during some of this period, which he gratefully acknowledges.

\section{The Compositional shuffle conjecture}
\subsection{Plethystic operators}
A $\lambda$-ring is a ring $R$ with a family of ring endomorphisms $(p_i)_{i\in \Z_{>0}}$ satisfying 
\[p_1[x]=x,\quad p_m[p_n[x]]=p_{mn}[x],\quad (x\in R,\quad m,n\in\Z_{>0}).\] 
Unless stated otherwise the endomorphisms are defined by $p_n(x)=x^n$ for each variable $x$ such as $q,t,u,v,z,x_i,y_i$.
The ring of symmetric functions over the $\lambda$-ring $\Q(q,t)$
is a free $\lambda$-ring with 
generator $X=x_1+x_2+\cdots$, and will be denoted $\Sym[X]$.
We will employ the standard notation used for
plethystic substitution defined as follows: 
given an element $F\in \Sym[X]$ and $A$ in some $\lambda$-ring $R$,
the plethystic substitution $F[A]$ is the image of the homomorphism
from $\Sym[X]\rightarrow R$ defined by replacing $p_n$ by $p_n(A)$.
%of replacing every indeterminant $x$ appearing $A$ by $x^k$,
%including the variables $x_i$ appearing in $X$. We then
%extend this operation to any symmetric function by deciding that
%it is a homomorphism. 
For instance, we would have
\[p_1 p_2[X/(1-q)]=p_1[X] p_2[X](1-q)^{-1}(1-q^2)^{-1}.\]
See \cite{haiman2001polygraphs} for a reference.

%Let $R$ be the $\lambda$-ring of rational functions of $q,t$ over $\Q$.
%The ring of symmetric functions over $R$ is denoted by $\Sym[X]$. It is the free% $\lambda$-ring with generator $X=x_1+x_2+\ldots$.

The modified Macdonald polynomials \cite{garsia1998explicit} will be 
denoted 
\[\Ht_\mu=t^{n(\mu)} J_{\mu}[X/(1-t^{-1});q,t^{-1}] \in\Sym[X]\]
where $J_\mu$ is the integral form of the Macdonald polynomial
\cite{Mac}, and
\[n(\mu)=\sum_{i} (i-1)\mu_i.\]
The operator $\nabla:\Sym[X]\to\Sym[X]$ is defined by
\begin{equation}
\label{nabdef}
\nabla \Ht_\mu = \Ht_\mu[-1] \Ht_\mu = (-1)^{|\mu|} q^{n(\mu')} t^{n(\mu)} \Ht_\mu.
\end{equation}
Note that our definition differs from the usual one from
\cite{Bergeron99identitiesand}
by the sign $(-1)^{|\mu|}$.
We also have the sequences of operators 
$B_r, C_r:\Sym[X]\rightarrow\Sym[X]$ given by the following formulas:
$$
(B_r F)[X] = F[X-(q-1)z^{-1}] \pExp[-z X]\big|_{z^r},
$$
$$
(C_r F)[X] = -q^{1-r} F[X+(q^{-1}-1)z^{-1}] \pExp[z X]\big|_{z^r},
$$
where $\pExp[X]=\sum_{n=0}^\infty h_n[X]$ is the plethystic exponential and $\vert_{z^r}$ denotes the operation of taking the coefficient of $z^r$ of a Laurent power series. Our definition again 
differs from the one in \cite{haglund2012compositional} by a factor $(-1)^r$.
For any composition $\alpha$, let $C_\alpha$ denote the composition $C_{\alpha_1}\cdots
C_{\alpha_l}$, and similarly for $B_\alpha$.

Finally we denote by $x \mapsto \bar{x}$ the involutive
automorphism of $\Q(q,t)$ obtained by sending $q,t$ to $q^{-1},t^{-1}$. We denote by $\omega$ the $\lambda$-ring automorphism of $\Sym[X]$ obtained by sending $X$ to $-X$ and by $\bar\omega$ its composition with $\bar{*}$, i.e.
\[
(\omega F)[X] = F[-X],\quad (\bar\omega F)[X] = \bar F[-X].
\]

\subsection{Parking functions}
We now recall the combinatorial background to state the shuffle
conjecture, for which we refer to Haglund's book \cite{haglund2008catalan}.
We consider the infinite grid on the top right quadrant of the plane. Its intersection points are denoted as $(i,j)$ with $i,j\in\Z$. For each cell of the grid its coordinates $(i,j)$ are the coordinates of the top right corner. Thus $i=1,2,\ldots$ indexes the columns and $j=1,2,\ldots$ indexes the rows.
Let $\DD$ be the set of Dyck paths of all lengths. A Dyck path of length $n$ is a grid path from $(0,0)$ to $(n,n)$ consisting of North and East steps that stays
above the main diagonal $i=j$. For $\pi\in\DD$ denote by $|\pi|$ its
length $n$. 
For $\pi\in \DD$, let
\[\area(\pi):=\#\Area(\pi),\quad
\Area(\pi):=\left\{(i,j): i<j,\ (i,j) \mbox{ under } \pi\right\}.\]
%If $(i,j)\in \Area(\pi)$, we will say that $i$ and $j$ attack each
%other. MEH.
This is the set of cells between the path and the diagonal.
Let $a_j$ denote the number of cells $(i,j) \in\Area(\pi)$ in the row $j$. The \emph{area sequence} is the sequence $a(\pi)=(a_1,a_2,\ldots,a_n)$ and we have $\area(\pi)=\sum_{j=1}^n a_n$.

Let $(x_1,1), (x_2,2),\ldots, (x_n, n)$ be the cells immediately to the right of the North steps. The sequence $x(\pi)=(x_1,x_2,\ldots,x_n)$ is called the \emph{coarea sequence} and we have $a_j+x_j=j$ for all $j$.

We have the $\dinv$ statistic and the $\Dinv$ set defined by 
\[\dinv(\pi):=\#\Dinv(\pi),\quad \Dinv(\pi):= \Dinv^0(\pi) \cup \Dinv^1(\pi)=\]
\[\left\{(j,j'):1\leq j<j'\leq n,\ a_j=a_{j'}\right\}\cup
\left\{(j,j'):1\leq j'<j\leq n,\ a_{j'}=a_j+1\right\}.\]
For $(j,j')\in\Dinv(\pi)$ we say that $(x_j, j)$ \emph{attacks} $(x_{j'}, j')$.

For any $\pi$, the set of \emph{word parking functions}
associated to $\pi$ is defined by
\[\mathcal{WP}_\pi:=\left\{{w} \in \Z_{>0}^n :
w_{j}>w_{j+1} \;\mbox{whenever}\; x_{j}=x_{j+1}\right\}.\]
In other words, the elements of $\mathcal{WP}_\pi$
are $n$-tuples $w$ of positive integers which, when
written from bottom to top to the right of each North step,
are strictly decreasing on cells such that one is on top
of the other. For any $w$, let 
\[\dinv(\pi,w):=\#\Dinv(\pi,w),\quad 
\Dinv(\pi,w):=\left\{(j,j') \in \Dinv(\pi):w_j>w_{j'}\right\}.\]

We note that both of these conditions differ from the usual notation
in which parking functions are expected to increase rather than
decrease, and in which the inequalities are reversed in the definition
of $\dinv$. This corresponds to choosing the opposite total ordering on
$\Z_{>0}$ everywhere, which does not affect the final answer,
and is more convenient for the purposes of this paper.

Let us call $\alpha=(\alpha_1,...,\alpha_k)=\touch(\pi)$
the \emph{touch composition} of $\pi$ if $\alpha_1,...,\alpha_k$ are the
lengths of the gaps between the points where $\pi$ touches the 
main diagonal starting at the lower left. Equivalently, $\sum_{i=1}^k \alpha_i=n$ and the numbers $1$, $1+\alpha_1$, $1+\alpha_1+\alpha_2$, \ldots $1+\alpha_1+\cdots+\alpha_{k-1}$ are the positions of $0$ in the area sequence $a(\pi)$.

\begin{example}
\label{dyckex}
Let $\pi$ be the following Dyck path of length $8$ described in Figure \ref{dyckpic}.
Then we have 
\[\Area(\pi)=\left\{
(2,3),(2,4),(3,4),(3,5),(3,6),(4,5),(4,6),(5,6),(7,8)\right\},\]
\[\Dinv(\pi)=\left\{(1,2),(1,7),(2,7),(3,8),(4,5)\right\} \cup
\left\{(7,3),(8,4),(8,5)\right\},\]
\[\touch(\pi)=(1,5,2),\quad a(\pi)=(0,0,1,2,2,3,0,1),\]
\[x(\pi)=(1,2,2,2,3,3,7,7)\]
whence $\area(\pi)=9$, $\dinv(\pi)=5+3=8$.
The labels shown above correspond to the vector
$w=(9,5,2,1,5,2,3,2)$, which we can see is an element of
$\mathcal{WP}_\pi$ because we have $5>2>1$, $5>2$, $3>2$.
We then have
\[\Dinv(\pi,w)=\left\{(1,2),(1,7),(2,7)\right\}\cup
\left\{(7,3),(8,4)\right\},\]
giving $\dinv(\pi,w)=5$.

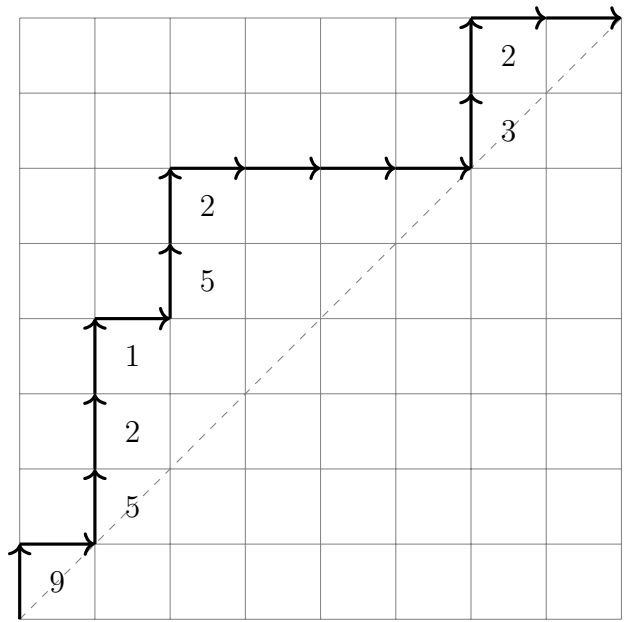
\begin{figure}
\begin{tikzpicture}
\draw[help lines] (0,0) grid (8,8);
\draw[dashed,color=gray] (0,0)--(8,8);
\draw[->,very thick] (0,0)--(0,1);
\draw[->,very thick] (0,1)--(1,1);
\draw[->,very thick] (1,1)--(1,2);
\draw[->,very thick] (1,2)--(1,3);
\draw[->,very thick] (1,3)--(1,4);
\draw[->,very thick] (1,4)--(2,4);
\draw[->,very thick] (2,4)--(2,5);
\draw[->,very thick] (2,5)--(2,6);
\draw[->,very thick] (2,6)--(3,6);
\draw[->,very thick] (3,6)--(4,6);
\draw[->,very thick] (4,6)--(5,6);
\draw[->,very thick] (5,6)--(6,6);
\draw[->,very thick] (6,6)--(6,7);
\draw[->,very thick] (6,7)--(6,8);
\draw[->,very thick] (6,8)--(7,8);
\draw[->,very thick] (7,8)--(8,8);
\node at (.5,.5) {9};
\node at (1.5,1.5) {5};
\node at (1.5,2.5) {2};
\node at (1.5,3.5) {1};
\node at (2.5,4.5) {5};
\node at (2.5,5.5) {2};
\node at (6.5,6.5) {3};
\node at (6.5,7.5) {2};
\end{tikzpicture}
\caption{Example of a Dyck path of size 8.}
\label{dyckpic}
\end{figure}
\end{example}

\subsection{The shuffle conjectures}
For any infinite set of variables $X=\{x_1,x_2,...\}$, let $x_w=x_{w_1}\cdots x_{w_n}$.
In this notation, the original shuffle conjecture \cite{haglund2005diagcoinv} 
states
\begin{conjecture}[\cite{haglund2005diagcoinv}] We have
\[
(-1)^n \nabla e_n=
\sum_{|\pi|=n} t^{\area(\pi)}\sum_{w\in \mathcal{WP}_\pi}
q^{\dinv(\pi,w)} x_{w}.
\]
In particular, the right hand side is symmetric in the $x_i$,
and in $q,t$.
\end{conjecture}
The stronger compositional shuffle conjecture \cite{haglund2012compositional}
states
\begin{conjecture}[\cite{haglund2012compositional}] For any composition $\alpha$, 
we have
\begin{equation} \label{compshuffeq}
(-1)^n \nabla C_{\alpha}(1)=
\sum_{\touch(\pi)=\alpha} t^{\area(\pi)}\sum_{w\in \mathcal{WP}_\pi}
q^{\dinv(\pi,w)} x_w.
\end{equation}
\end{conjecture}

\subsection{From $(\area,\dinv)$ to $(\bounce,\area')$}
\label{areadinvbouncesec}
In this paper, we will prove an equivalent version of this conjecture, as  obtained in \cite{loehr2014new}, Theorem 14 by applying the
$(\area,\dinv)$ to $(\bounce,\area')$ bijection from \cite{haglund2005conjectured},
\cite{haglund2008catalan}. We include our construction of this bijection because it seems to be different from the original one, and to demonstrate that it comes naturally from analysis of the attack relation. An important property of our construction is that it comes with a natural lift from Dyck paths to parking functions.
%To express the compositional shuffle conjecture of \cite{haglund2012compositional} in terms of the Dyck path characters we need to go through the so-called ``(area,dinv) to (bounce,area) map'' from \cite{haglund2008catalan}. This is a bijection on Dyck paths constructed as follows. Start with a Dyck path $w_0$ from $(0,0)$ to $(n,n)$. Consider the cells immediately to the right of each North step. Put them in the reading order: first put all the cells on the diagonal from left to right, then all the cells one step away from the diagonal from left to right, and so on. If the coordinates of the $i$-th cell in the resulting sequence are $(x_i, y_i)$ put $a_i=y_i-x_i$. Then the sequence of pairs $(a_i, x_i)$ is sorted alphabetically. We say that a pair $(a_i, x_i)$ attacks all pairs that are after $(a_i, x_i)$ and before $(a_i+1, x_i)$ in the alphabetic order. We define $w$ to be the unique Dyck path from $(0,0)$ to $(n,n)$ such that for $i<j$ $(a_i, x_i)$ attacks $(a_j, x_j)$ if and only if the cell $(i,j)$ is under $w$. It appears that the map sending $w_0$ to $w$ is a bijection.

From any pair $\pi\in\DD$, $w\in\mathcal{WP}_\pi$ we will obtain a pair $\pi'\in\DD$, $w'\in\mathcal{WP}'_{\pi'}$ by a procedure described below. After the end of this section we will only work with $\pi', w'$, so we will drop the apostrophe.

The Dyck path $w'$ is obtained as follows: sort the cells $(x_j, j)$ in the \emph{reading
order}, i.e. in increasing order by the corresponding labels $a_j$,
using the row index $j$ to break ties. Equivalently, we read the cells by diagonals from bottom to top, and from left to right in each diagonal.
%consider the cells immediately to the right of each North step, and sort
%them in reading order, i.e. start with all cells which are on the diagonal,
%then all cells which are one away from the diagonal and so on, using
%lexicographic order to break ties.
For instance, for the path $\pi$ from Example \ref{dyckex}, 
the list would be
\begin{equation}
\label{roex}
\left\{(1,1),(2,2),(7,7),(2,3),(7,8),(2,4),(3,5),(3,6)\right\}.
\end{equation}
Let $\sigma_j$ be the position of the cell $(x_j,j)$ in this list. This defines a permutation $\sigma \in S_n$.
In the example case, we would get
\[
\sigma=
\begin{pmatrix}1 & 2 & 3 & 4 & 5 & 6 & 7 & 8 \\
1 & 2 & 4 & 6 & 7& 8 & 3 & 5
\end{pmatrix}
\]
Now we observe that for each $j=1,\ldots,n$ the cell $(x_j, j)$ attacks all the subsequent cells in the reading order whose position is before the position where we would place $(x_j, j+1)$, if it were an element of the list. 
%Note that $(x_j, j+1)$ does not have to be in our sequence, so we are talking about the place it \emph{would} take according to the reading order if it was there. 

More precisely, there is a unique Dyck path $\pi'$ for which 
\[\Area(\pi')=\sigma(\Dinv(\pi))=
\left\{(\sigma_j,\sigma_{j'}):(j,j')\in \Dinv(\pi)\right\}.\]
%\left\{(\sigma_j,\sigma_i):(i,j) \in \Dinv^1(\pi)\right.\}.\]
%is the set obtained by applying $\sigma$ the elements
%of $\Dinv(\pi)$, =\sigma\left(\Dinv(\pi)\right)$.
The map $\pi\to \pi'$ is the desired bijection. To see the bijectivity one can either use \cite{haglund2008catalan} or see Remark \ref{rem:touch} below.

If $\pi$ is the Dyck path from our example, then $\pi'$ would be given by
the path in Figure \ref{bijdyck}.
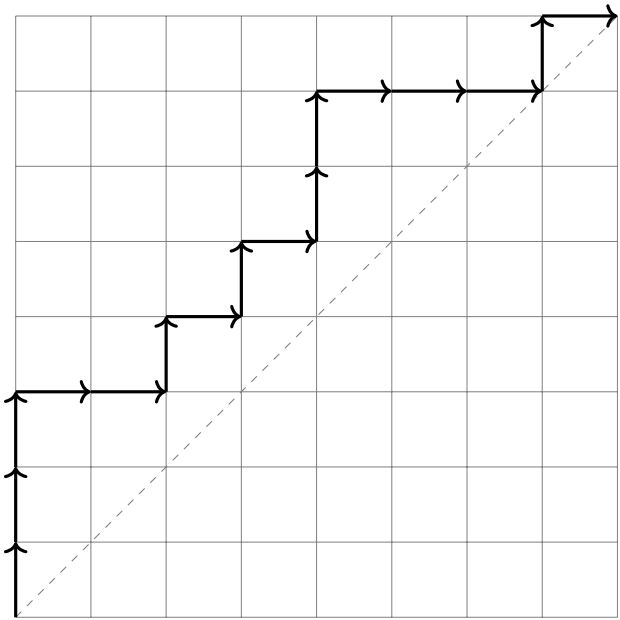
\begin{figure}
\begin{tikzpicture}
\draw[help lines] (0,0) grid (8,8);
\draw[dashed,color=gray] (0,0)--(8,8);
\draw[->,very thick] (0,0)--(0,1);
\draw[->,very thick] (0,1)--(0,2);
\draw[->,very thick] (0,2)--(0,3);
\draw[->,very thick] (0,3)--(1,3);
\draw[->,very thick] (1,3)--(2,3);
\draw[->,very thick] (2,3)--(2,4);
\draw[->,very thick] (2,4)--(3,4);
\draw[->,very thick] (3,4)--(3,5);
\draw[->,very thick] (3,5)--(4,5);
\draw[->,very thick] (4,5)--(4,6);
\draw[->,very thick] (4,6)--(4,7);
\draw[->,very thick] (4,7)--(5,7);
\draw[->,very thick] (5,7)--(6,7);
\draw[->,very thick] (6,7)--(7,7);
\draw[->,very thick] (7,7)--(7,8);
\draw[->,very thick] (7,8)--(8,8);
\end{tikzpicture}
\caption{Image of the path from Figure \ref{dyckpic} under the
$(\area,\dinv)$ to $(\bounce,\area')$  bijection.}
\label{bijdyck}
\end{figure}

The above statistics can be translated into new statistics under this
bijection. First, it is clear from the construction that
$\dinv(\pi)=\area(\pi')$. 
We next explain how to calculate $\area(\pi)$ from $\pi'$:
for any path, we obtain a new Dyck path called the ``bounce path'' as  
follows: start at the origin $(0,0)$, and begin moving North until
contact is made with the first East step of $\pi$. Then start moving East until  
contacting the diagonal. Then move North until contacting the path again,  
and so on. Note that contacting the path means running into the left  
endpoint of an East step, but passing by the rightmost endpoint does  
not count, as illustrated below. The bounce path splits the main diagonal into 
the \emph{bounce blocks}. We number the bounce blocks starting from $0$ and define the 
\emph{bounce sequence} $b(\pi)=(b_1, b_2, \ldots, b_n)$ in such a way that for any $i$ the cell $(i,i)$ belongs to the $b_i$-th block.
We then define  
\[\bounce(\pi'):=\sum_{i=1}^n b_i.\]
Another way to describe this construction is to say that $b_1=0$, $b_{i+1}\in\{b_i, b_i+1\}$ and if $i, i'$ are the smallest indices
for which $b_i=c$ and $b_{i'}=c+1$ for some $c$, then $i'$ is the smallest index with $i'>i$ such that $(i,i')\notin \Area(\pi')$. This description and $\Area(\pi')=\sigma(\Dinv(\pi))$ implies $b_{\sigma_i} = a_i$, hence $\bounce(\pi')=\area(\pi)$. See \cite{haglund2008catalan} for an alternative treatment.

For the path $\pi'$ above, the bounce path is shown in Figure \ref{bouncedyck}
with the original path in gray. The bounce sequence is given by the numbers written under the diagonal. We have 
\[b(\pi')=(0,0,0,1,1,2,2,3),\quad \bounce(\pi')=9=\area(\pi).\]
\begin{figure}
\begin{tikzpicture}
\draw[help lines] (0,0) grid (8,8);
\draw[dashed,color=gray] (0,0)--(8,8);
\draw[->,very thick,color=gray] (0,0)--(0,1);
\draw[->,very thick,color=gray] (0,1)--(0,2);
\draw[->,very thick,color=gray] (0,2)--(0,3);
\draw[->,very thick,color=gray] (0,3)--(1,3);
\draw[->,very thick,color=gray] (1,3)--(2,3);
\draw[->,very thick,color=gray] (2,3)--(2,4);
\draw[->,very thick,color=gray] (2,4)--(3,4);
\draw[->,very thick,color=gray] (3,4)--(3,5);
\draw[->,very thick,color=gray] (3,5)--(4,5);
\draw[->,very thick,color=gray] (4,5)--(4,6);
\draw[->,very thick,color=gray] (4,6)--(4,7);
\draw[->,very thick,color=gray] (4,7)--(5,7);
\draw[->,very thick,color=gray] (5,7)--(6,7);
\draw[->,very thick,color=gray] (6,7)--(7,7);
\draw[->,very thick,color=gray] (7,7)--(7,8);
\draw[->,very thick,color=gray] (7,8)--(8,8);
\draw[->,very thick] (0,0)--(0,1);
\draw[->,very thick] (0,1)--(0,2);
\draw[->,very thick] (0,2)--(0,3);
\draw[->,very thick] (0,3)--(1,3);
\draw[->,very thick] (1,3)--(2,3);
\draw[->,very thick] (2,3)--(3,3);
\draw[->,very thick] (3,3)--(3,4);
\draw[->,very thick] (3,4)--(3,5);
\draw[->,very thick] (3,5)--(4,5);
\draw[->,very thick] (4,5)--(5,5);
\draw[->,very thick] (5,5)--(5,6);
\draw[->,very thick] (5,6)--(5,7);
\draw[->,very thick] (5,7)--(6,7);
\draw[->,very thick] (6,7)--(7,7);
\draw[->,very thick] (7,7)--(7,8);
\draw[->,very thick] (7,8)--(8,8);
\node at (.7,.3) {0};
\node at (1.7,1.3) {0};
\node at (2.7,2.3) {0};
\node at (3.7,3.3) {1};
\node at (4.7,4.3) {1};
\node at (5.7,5.3) {2};
\node at (6.7,6.3) {2};
\node at (7.7,7.3) {3};
\end{tikzpicture}
\caption{The bounce path of the path in Figure \ref{bijdyck}.}
\label{bouncedyck}
\end{figure}
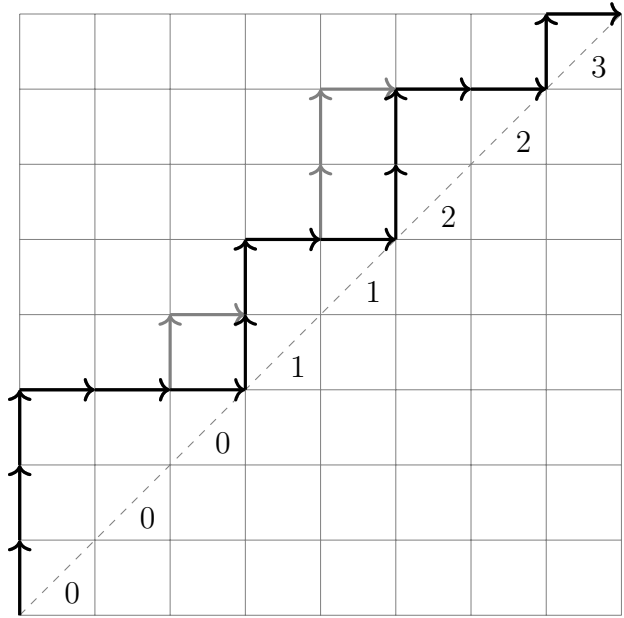

Next, we show how to reconstruct $\touch(\pi)$ from $\pi'$. 
For any path $\pi'$ of length $n$, let $l$ be the number of 
North steps from $(0,0)$
until the first East step, which is the same as the length of the first 
bounce block. Let $\tilde{\pi}$ be the part of the path such that 
$\pi'=N^lE\tilde{\pi}$, the result of beginning with $l$ North steps 
starting at the origin, followed by an East step, followed by 
the contents of $\tilde{\pi}$. 
%adding one East step and $l$
%North steps 
Define numbers $t_i$ by 
\[t_i:=\bounce\left(N^{i+1}EN^{l-i}E \tilde{\pi}\right),\quad 
0\leq i \leq l.\]
Note that the path $N^{i+1}EN^{l-i}E \tilde{\pi}$ has length $n+1$ for each $i$, and we have 
$t_0=n+\bounce(\pi')$ and then $t_i$ go down to $t_l=\bounce(\pi')$.
Define 
\[
\touch'(\pi'):=(t_0-t_{1},...,t_{l-1}-t_{l}).
\]
\begin{prop}
For every Dyck path $\pi$
\[\touch'(\pi')=\touch(\pi).\]
\end{prop}
For instance, in the example above we
would have $l=3$,
\[(t_0,t_1,t_2,t_3)=(17,16,11,9),\quad
\touch'(\pi')=\touch(\pi)=(1,5,2).\]
\begin{proof}
Consider the $i$-th touch point $(x,x)$ of $\pi$ (we count the touch points starting from $0$, i.e. $(0,0)$ is the $0$-th touch point.) It splits $\pi$ into two parts: $\pi_1$ followed by $\pi_2$. Construct a new path $\hat{\pi}$ of length $n+1$ by taking a step North, then following a translated copy of $\pi_2$, then taking a step East, then following a translated copy of $\pi_1$. The new path has length $n+1$ and its area is bigger than the area of $\pi$ by $n-x$. The (area,dinv) to (bounce,area$'$) map applied to $\hat{\pi}$ gives precisely the path $N^{i+1}EN^{l-i}E \tilde{\pi}$. Thus we have
\[
\area(\hat{\pi}) = \bounce(N^{i+1}EN^{l-i}E \tilde{\pi}) = t_i,
\]
\[
\area(\hat{\pi}) = n-x+\area(\pi) = n-x+\bounce(\pi').
\]
So the sizes of the gaps between the touch points of $\pi$ are exactly the differences $t_{i-1}-t_i$.
\end{proof}
\begin{rem}\label{rem:touch}
The construction we have used in the proof above can also be used to prove the bijectivity of the (area,dinv) to (bounce,area$'$) map. Here is an idea of a proof. First, every Dyck path arises as $\hat{\pi}$ above for unique $\pi$ and $i$. On the other hand, every Dyck path can be uniquely written as $N^{i+1}EN^{l-i}E \tilde{\pi}$. Thus iterating the construction we obtain every Dyck path on each side of the (area,dinv) to (bounce,area$'$) map in a unique way.
\end{rem}

Having analyzed the statistics associated to a Dyck path we turn to the analysis of what happens to word parking functions.
The $\dinv$ statistic is straightforward.
For any $w'\in \Z_{>0}^n$, let
\[\inv(\pi',w'):=\#\Inv(\pi',w'),\quad 
\Inv(\pi',w'):=\left\{(i,j) \in \Area(\pi'),\ w'_i>w'_j\right\},\] 
so that
\[\Inv(\pi',w')=\sigma\left(\Dinv(\pi,w)\right),\quad
w'_{\sigma_i}=w_{i}.\] 
For the value of $w$ from Example \ref{dyckex}, we would have
\[w'=(9,5,3,2,2,1,5,2),\quad
%(6,1,4,2,4,3,4,2)\]
%\[\Area(\pi')=\left\{(1,2),(1,3),(2,3),(3,4),(4,5),(5,6),(5,7),(6,7)\right\}\]
\Inv(\pi',w')=\{(1,2),(1,3),(2,3),(3,4),(5,6)\}.\]
In particular, $\inv(\pi',w')=\dinv(\pi,w)=5$.

Finally, we reconstruct the word parking function condition.
A cell $(i,j)$ is called a {\em corner} of $\pi'$
if it is above the path, but both its Southern and Eastern neighbors 
are below the path. Denote the set of corners by $c(\pi')$. 
One can check that the corners of $\pi'$ correspond to pairs of cells with one on top 
of the other in $\pi$. For instance, from our example we have 
%the corners of $\pi'$ from the 
$c(\pi')=\{(2,4),(3,5),(4,6),(7,8)\}$. 
More precisely, we have
\[
c(\pi') := \{(\sigma_j, \sigma_{j+1}):\;1\leq j<n,\;x_j=x_{j+1}\}.
\] 
%These are 
%precisely the pairs of indices $(i,j)$ for which the corresponding 
%elements in \eqref{roex} are stacked with one above the other.
We therefore define 
\begin{equation}
\label{WP1}
\mathcal{WP}'_{\pi'}:=\left\{w' \in \Z_{>0}^n :
w'_i>w'_j \mbox{ for } (i,j) \in c(\pi')\right\},
\end{equation}
so that the condition $w\in \mathcal{WP}_\pi$ is equivalent to $w'\in \mathcal{WP}'_{\pi'}$.

Putting this together, we have
\begin{prop}
For any composition $\alpha$ we have
\begin{equation}\label{eq:Dalpha}
D_{\alpha}(q,t)=
\sum_{\touch'(\pi)=\alpha} t^{\bounce(\pi)}\sum_{w \in \mathcal{WP}'_\pi} 
q^{\inv(\pi,w)} x_w,
\end{equation}
where $D_\alpha(q,t)$ is the right hand side of \eqref{compshuffeq}.
\end{prop}

\section{Characteristic functions of Dyck paths}
\subsection{Simple characteristic function}
We are going to study the summand in $D_{\alpha}(q,t)$ as a 
function of $\pi$. It is convenient to first introduce a simpler object where we drop the assumption $w\in\mathcal{WP}'_\pi$ and instead sum over all labellings. Given a Dyck path of length $n$, define $\chi(\pi)\in \Sym[X]$ as follows:
\begin{defn}\label{defn.chi}
$$
\chi(\pi): = \sum_{w \in \Z_{>0}^{n}} q^{\inv(\pi, w)} x_w.
$$
\end{defn}
If $i<j$ and $(i,j)$ is under $\pi$, i.e. $(i,j)\in\Area(\pi)$ we say that $i$ and $j$ \emph{attack} each other. It is not obvious from the definition that $\chi(\pi)$ defines a symmetric function.
In fact, as we point out in Remark \ref{chilltrem}, $\chi(\pi)$ is actually an example of an LLT polynomial,
but we present a proof in our setup here:
\begin{prop}\label{prop.sym}
The expression for $\chi(\pi)$ above is symmetric in the variables $x_1, x_2, x_3,\ldots$, so that Definition \ref{defn.chi} correctly defines an element of $\Sym[X]$.
\end{prop}
\begin{proof}
We take the main idea from the proof of Lemma 10.2 from \cite{haglund2005combinatorial}. First note that for each $n$ the correspondence $\pi \to \Area(\pi)$ is a bijection between the set of Dyck paths of length $n$ and the set of subsets $R\subset\{(j,j'):1\leq j<j'\leq n\}$ satisfying the following property:
\begin{equation*}\tag{*}
\text{if $j<j'<j''$ and $(j,j'')\in R$, then both $(j,j')$ and $(j',j'')$ are in $R$.}
\end{equation*}
In the proof we will work with $R$ instead of $\pi$ and write $\chi(R,n)$, $\inv(R,w)$ instead of $\chi(\pi)$, $\chi(\pi,w)$. For each subset $S\subset\{1,2,\ldots,n\}$, $S=\{s_1<s_2<\cdots<s_{\#S}\}$ let $R_S=\{(j,j'):\,(s_j,s_{j'})\in R\}$. Then $R_S$ again satisfies (*). 

It is enough to show that $\chi(R,n)$ is unaffected by interchange of $x_i$ and $x_{i+1}$ for any two neighboring indices $i, i+1$. For each subset $S\subset\{1,2,\ldots,n\}$ and a function $f:\{1,\ldots,n\}\to \Z_{>0}\setminus\{i,i+1\}$ let $\chi_{S,f}$ be the sub-sum of $\chi(R,n)$ corresponding to sequences $w$ where the set of positions of $i$ and $i+1$ in $w$ is $S$, and the values of $w$ outside of $S$ are given by $f$. It is enough to show that $\chi_{S,f}$ is symmetric in $x_i$ and $x_{i+1}$. We have
\[
\inv(R,w) = \#\{(j,j')\in R_S: w_{s_j}=i+1,\,w_{s_{j'}}=i\} + \inv_{S,f},
\]
where $\inv_{S,f}$ depends only on $S$ and $f$, but doesn't depend on the positions of $i$ or $i+1$ in $S$. Thus we have
\[
\chi_{S,f} =  \chi(R_S,\#S)(x_i,x_{i+1}) \; q^{\inv_{S,f}} \prod_{j\notin S} x_{f(j)},
\]
where
\[
\chi(R,k)(x_1,x_2) = \sum_{w\in\{1,2\}^k} q^{\inv(R, w)} x_w.
\]

So it is enough to show that $\chi(R,k)(x_1,x_2)$ is symmetric in $x_1$ and $x_2$ for any $k$, $R$ satisfying (*). We proceed by induction on the size of $R$, the base case $\#R=0$ being trivial. Fix $k$ and $R\neq \varnothing$ and pick $(a,b)\in R$ maximal in the sense that $(a,j)\notin R$ for all $j>b$ and $(j,b)\notin R$ for all $j<a$. Then $R'=R\setminus\{(a,b)\}$ satisfies (*). Consider the difference
\[
\chi(R,k) - \chi(R',k) = \sum_{w\in\{1,2\}^k} (q^{\inv(R, w)}-q^{\inv(R', w)}) x_w.
\]
The coefficient $(q^{\inv(R, w)}-q^{\inv(R', w)})$ is non-zero only if $w_a=2$ and $w_b=1$. If this happens, then $\inv(R,w)=\inv(R',w)+1$. Consider contributions of pairs of the form $(a,j)$, $(b,j)$, $(j,a)$, $(j,b)$ to $\inv(R',w)$. Since $w_a=2$ pairs $(j,a)$ do not contribute anything, and pairs $(a,j)$ contribute $1$ precisely when $a<j<b$ and $w_j=1$. Since $w_b=1$ pairs $(b,j)$ do not contribute, and pairs $(j,b)$ contribute $1$ precisely when $a<j<b$ and $w_j=2$. The net contribution is the number of $j$ such that $a<j<b$, which is $b-a-1$. We see that $\inv(R',w)=b-a-1 + \inv(R'_{S},w_S)$, where $S=\{1,\ldots,k\}\setminus\{a,b\}$, $w_S$ denotes the sequence $w$ with the entries $w_a$ and $w_b$ removed. Thus we obtain
\[
\chi(R,k) - \chi(R',k) = (q-1) x_1 x_2 \chi(R'_S,k-2).
\]
By the induction hypothesis $\chi(R',k)$ and $\chi(R'_S,k-2)$ are symmetric in $x_1, x_2$. Hence $\chi(R,k)$ is also symmetric.
\end{proof}

Another way to formulate this property is as follows: for a composition $c_1+c_2+\cdots+c_k=n$ consider the multiset $M_c = 1^{c_1} 2^{c_2} \ldots k^{c_k}$. Consider the sum
$$
\sum_{w \text{\; a permutation of $M_c$}} q^{\inv(\pi, w)}.
$$
Proposition \ref{prop.sym} simply says that this sum does not depend on the order of the numbers $c_1, c_2, \ldots, c_k$, or equivalently on the linear order on the set of labels. If $\lambda$ is the partition with components $c_1,c_2,\ldots,c_k$, then this sum computes the coefficient of the monomial symmetric function $m_\lambda$ in $\chi(\pi)$,
so we have (set $h_c=h_{c_1}\cdots h_{c_k}$)
\begin{equation}\label{eq.sym}
(\chi(\pi), h_c) = \sum_{w \text{\; a permutation of $M_c$}} q^{\inv(\pi, w)}.
\end{equation}

We list here a few properties of $\chi$ so that the reader has a feeling of what kind of object it is.

For a Dyck path $\pi$ denote by $\pi^{op}$ the reversed Dyck path, i.e. the path obtained by replacing each North step by East step and each East step by North step and reversing the order of steps. Reversing also the order of the components of $c$ in (\ref{eq.sym}) we see 

\begin{prop}
$$
\chi(\pi) = \chi(\pi^{op}).
$$
\end{prop}

Proofs of the following two statements are essentially taken from \cite{haglund2005combinatorial}.
\begin{prop}\label{prop.omega}
$$
\bar\omega \chi(\pi) = (-1)^{|\pi|} q^{-\area(\pi)} \chi(\pi).
$$
\end{prop}

\begin{prop}\label{prop.q1}
$$
\chi(\pi)[(q-1)X] = (q-1)^{|\pi|} \sum_{w \in \Z_{>0}^{|\pi|}\; \text{no attack}} q^{\inv(\pi, w)} x_w,
$$
where ``no attack'' means that the summation is only over vectors $w$ such that $w_i\neq w_j$ for $(i,j)\in\Area(\pi)$.
\end{prop}
\begin{proof}[Proofs]
We follow Chapter 4 of \cite{haglund2005combinatorial}. For an integer $n$ and a subset $D\subset\{1,\ldots,n-1\}$ Gessel's \emph{quasisymmetric function} $Q_{n,D}$ in $x=(x_1,x_2,\ldots)$ is given by
\[
Q_{n,D}(x) = \sum_{\substack{w_1\leq \cdots \leq w_n \\ w_i=w_{i+1} \Rightarrow i\notin D}} x_w.
\]
For each sequence $w\in\Z_{>0}^n$ its \emph{standardization} is the unique permutation $\Std(w)\in S_n$ such that
\[
w_i<w_j\;\text{or}\;(w_i=w_j\;\text{and}\;i<j) \Leftrightarrow \Std(w)_i<\Std(w)_j.
\]
In other words, $\Std(w)$ sorts pairs $(w_i,i)$ in lexicographic order. We notice the following properties:
\begin{equation}\label{eq:std}
\Inv(\pi, w) = \Inv(\pi,\Std(w)),\qquad \sum_{w:\Std(w)=\sigma} x_w = Q_{n,\Des(\sigma^{-1})}(x)\qquad(\sigma\in S_n),
\end{equation}
where $\Des(\sigma) = \{i:\sigma_i>\sigma_{i+1}\}$ is the descent set of $\sigma$. Thus the sum $\chi(\pi)$ splits as follows:
\[
\chi(\pi) = \sum_{\sigma\in S_n} q^{\inv(\pi, \sigma)} Q_{n,\Des(\sigma^{-1})}.
\]
Since $\chi(\pi)$ is symmetric by Proposition \ref{prop.sym}, we can apply Proposition 4.2 in \cite{haglund2005combinatorial}. Let $\cA$ be the ``super'' alphabet 
\[
\cA=\Z_+\cup \Z_-=\{1,2,3,\ldots,\bar{1},\bar{2},\bar{3},\ldots\}
\]
consisting of \emph{positive letters} $i\in Z_+$ and \emph{negative letters} $\bar{i}$. Let 
\[
z_i = x_i\quad(i\in\Z_+),\qquad z_{\overline{i}} = -y_i \quad(\bar{i}\in\Z_-).
\]
Then we have the following expression for $X=\sum_i x_i$, $Y=\sum_i y_i$:
\[
\chi(\pi)[X-Y] = \sum_{\sigma\in S_n} q^{\inv(\pi, \sigma)} \tilde Q_{n,\Des(\sigma^{-1})}(x,y),
\]
where
\[
\tilde Q_{n,D} = \sum_{\substack{w_1\leq \cdots \leq w_n \\ w_i=w_{i+1},\,w_i\in\Z_+ \Rightarrow i\notin D \\ w_i=w_{i+1},\,w_i\in\Z_- \Rightarrow i\in D}} z_w,
\]
and the summation is over the sequences of elements of $\cA$. The statement holds for an arbitrary choice of total ordering on $\cA$. We work with the following ordering:
\[
1<\bar{1}<2<\bar{2}<\cdots.
\]
We extend the definitions of $\Std$, $\Inv$, $\inv$ to sequences of elements of $\cA$ as follows:
\[
w_i<w_j\;\text{or}\;(w_i=w_j\in \Z_+\;\text{and}\;i<j) \;\text{or}\;(w_i=w_j\in \Z_-\;\text{and}\;i>j) 
\]
\[
\Leftrightarrow \Std(w)_i<\Std(w)_j,
\]
\[
\inv(\pi,w):=\#\Inv(\pi,w),\quad 
\Inv(\pi,w):=\left\{(i,j) \in \Area(\pi):\,w_i>w_j\;\text{or}\; w_i=w_j\in\Z_-\right\},
\]
so that the properties \eqref{eq:std} are satisfied. Therefore we have
\begin{equation}\label{eq:chi difference}
\chi(\pi)[X-Y] = \sum_{w\in\cA^n} q^{\inv(\pi,w)} z_w.
\end{equation}

Setting $X=0,Y=-X$ we obtain
$$
\chi(\pi)[-X] = (-1)^n \sum_{w\in\Z_{>0}^n} q^{\inv'(\pi, w)} x_w,
$$
where $\inv'(\pi, w)$ is the number of non-strict inversions of $w$ under the path,
\[\inv'(\pi,w):=\#\left\{(i,j) \in \Area(\pi),\ w_i\geq w_j\right\}.\] 
Reversing the order of labels we have
$$
\chi(\pi)[-X]  = (-1)^{|\pi|} \sum_{w\in\Z_{>0}^n} q^{\area(\pi) - \inv(\pi, w)} x_w,
$$
which implies Proposition \ref{prop.omega}.

To prove Proposition \ref{prop.q1} we set $X=qX$, $Y=X$ in \eqref{eq:chi difference}. Applying the involution from the proof of Lemma 5.1 in \cite{haglund2005combinatorial} (flipping the sign of the last label that attacks a label with the same absolute value) we see that the terms for $w\in\cA^n$ such that $|w_i|=|w_j|$ for some $(i,j)\in\Area(\pi)$ cancel out. In the remaining terms we have $|w_i|\neq |w_j|$ whenever $(i,j)\in\Area(\pi)$. Therefore the comparison between $w_i$ and $w_j$ depends only on $|w_i|$ and $|w_j|$. So we can first sum over sequences in $\Z_{>0}$, and then over the choices of signs. The latter summation produces an overall factor of $(q-1)^n$ and we obtain Proposition \ref{prop.q1}.
\end{proof}

\subsection{Weighted characteristic function}
To study the summand of $D_{\alpha}(q,t)$ in (\ref{eq:Dalpha}) as a 
function of $\pi$ we introduce a more general characteristic function.
%A weighted Dyck path is a Dyck path $w\in\DD$ together with a map $\mathrm{wt}:c(w)\to R$. If weights are not specified, they are assumed to be $1$. Thus every Dyck path is a weighted Dyck path. The set of weighted Dyck paths is denoted by $\WD$. 
Given a function $\mathrm{wt}:c(\pi)\to R$ on the set of corners of
some Dyck path $\pi$ of size $n$, let
\begin{equation}
\label{chiwtdef}
\chi(\pi, \mathrm{wt}) := \sum_{w \in \Z_{>0}^{n}} q^{\inv(\pi,w)} 
\left(\prod_{(i,j)\in c(\pi),\;w_i\leq w_j} \mathrm{wt}(i,j) \right)x_w,
\end{equation}
so in particular $(\ref{eq:Dalpha})$ becomes
\[D_{\alpha}(q,t)=\sum_{\touch'(\pi)=\alpha} t^{\bounce(\pi)} \chi(\pi,0).\]
For a constant function $\mathrm{wt}=1$ we recover the simpler characteristic function
\begin{equation}
\label{chidef}
\chi(\pi,1)=\chi(\pi).
\end{equation}
%In other words, we count inversions under the path with weight $q$, as before, but we also count non-inversions in the corne%rs with weights $\mathrm{wt}$. Clearly if $\mathrm{wt}=1$ we obtain the old character: $\chi(w, 1)=\chi(w)$. 
It turns out that we can express the weighted characteristic function $\chi(\pi, \mathrm{wt})$ in terms the unweighted one evaluated at different paths. In particular this implies that $\chi(\pi, \mathrm{wt})$ is symmetric too.

\begin{rem}
\label{chilltrem}
If $\pi'$ is the image of $\pi$ under the bijection from section 
\ref{areadinvbouncesec}, then we have that 
$\chi(\pi',0)=F_{\pi}(X;q)$, where $F_{\pi}(X;q)$ are the \emph{path symmetric
functions} from page 95 of Haglund's book \cite{haglund2008catalan}. As Haglund explains,
these functions are examples of LLT polynomials of vertical strips, using
the description of Bylund and Haiman. In fact, $\chi(\pi',1)$ is also 
an example of an LLT polynomial, but for a disjoint union of single boxes:
\[\chi(\pi',1)=\LLT_{[a_n+1]/[a_n],...,[a_1]/[a_1]}(X;q),\]
where $(a_1,...,a_n)=a(\pi)$ is the area sequence.
\end{rem}

\begin{prop}
We have that $\chi(\pi,{\mathrm{wt}})$ is symmetric in the $x_i$
variables, and so defines an element of $\Sym[X]$.
\end{prop}

\begin{proof}
Let $\pi$ be a Dyck path, and let $(i,j)\in c(\pi)$ be one of its
corners. We denote by $\mathrm{wt}_1$ the weight on $\pi$ which is
obtained from $\mathrm{wt}$ by setting the weight of $(i,j)$ to
$1$. Let $\pi'$ be the Dyck path obtained from $\pi$ by turning the
corner inside out, in other words the Dyck path of smallest area 
which is both above $\pi$ and above $(i,j)$. Let $\mathrm{wt}_2$ be
the weight on $\pi'$ which coincides with $\mathrm{wt}$ on all corners
of $\pi'$ which are also corners of $\pi$ and is $1$ on other
corners. We claim that
\begin{equation}
\label{cornereq}
\chi(\pi, \mathrm{wt}) = \frac{q \mathrm{wt}(i,j) - 1}{q-1} \chi(\pi, \mathrm{wt}_1) + \frac{1 - \mathrm{wt}(i,j)}{q-1} \chi(\pi', \mathrm{wt}_2). 
\end{equation}
To see this, notice that if we group the terms on the right hand
side, then both sides may be written as a sum over vectors
$w\in\mathbb{Z}_{>0}$. Split both sums according to terms in
which $w_{i}>w_{j}$ resulting in an additional factor of $q$, or 
$w_i\leq w_j$ resulting in an additional weight factor. It is easy to
check that both sums agree on both the left and right sides.

The result now follows because we may recursively express any
$\chi(\pi,{\mathrm{wt}})$ in terms of $\chi(\pi)$, which we have
already remarked is symmetric.

%The contributions of labellings which have inversion at $(i,j)$ to the first summand is proportional to $1$, to the second summand is proportional to $q$, thus the overall contribution is proportional to $1$. In the case of no inversion at $(i,j)$ both summands give $1$, so the overall contribution is proportional to $\mathrm{wt}(i,j)$. 
\end{proof}

%In fact by successively turning the corners inside out (i.e. applying the formula above) we obtain a map $\mathbf{insideout}:\WD \to R[\DD]$ which is a section of the embedding $\DD \subset \WD$ and such that 
%$$
%\chi\circ\mathbf{insideout} = \chi. 
%$$

%Chapter 4 and Lemma 5.1 from \cite{haglund2005combinatorial} 
%apply in a straightforward way to our case, and can be used to
%show that
%%\begin{prop}\label{prop.q1}
%We have that
%\[\overline{\chi(\pi)} = (-1)^{|\pi|} q^{-\area(\pi)} \chi(\pi),\]
%\[\chi(\pi)[(q-1)X] = (q-1)^{|\pi|} 
%\sum_{w \in \Z_{>0}^{|\pi|} \text{no attack}} q^{\inv(\pi,w) } \prod_{i=1}^{|\pi|} {x_{w_i}},\]
%where ``no attack'' means that the summation is only over 
%vectors $w$ such that $w_i\neq w_j$ for $(i,j)$ under $\pi$. 
%\end{prop}

\begin{example}
\label{easydyck}
In particular, we can use this to extract $\chi(\pi,0)$ from
$\chi(\pi',1)$ for all $\pi'$. 
If $S\subset c(\pi)$ is any
subset of the set of corners, let $\pi_S\in \mathbb{D}$ denote the path
obtained by flipping the corners that are in $S$. Then 
equation \eqref{cornereq} implies that
\begin{equation}
\label{chi02chi1}
\chi(\pi,0)=(1-q)^{-|c(\pi)|} \sum_{S \subset c(\pi)} (-1)^{|S|}\chi(\pi_S,1).
\end{equation}
For instance, let $\pi$ be the Dyck path in Figure \ref{easydyckfig}
\begin{figure}
\begin{tikzpicture}
\draw[help lines] (0,0) grid (3,3);
\draw[dashed, color=gray] (0,0)--(3,3);
\draw[->,very thick] (0,0)--(0,1);
\draw[->,very thick] (0,1)--(0,2);
\draw[->,very thick] (0,2)--(1,2);
\draw[->,very thick] (1,2)--(2,2);
\draw[->,very thick] (2,2)--(2,3);
\draw[->,very thick] (2,3)--(3,3);
\end{tikzpicture}
\caption{}
\label{easydyckfig}
\end{figure}
Then setting $x_i=0$ for $i>3$ reduces formula
\eqref{chiwtdef} to a finite sum over 27 terms, from which we can deduce that
\[\chi(\pi)=m_3+(2+q)m_{21}+(3+3q)m_{111}=s_3+(1+q)s_{21}+qs_{111}.\]
Similarly, if $\pi'=\pi_{\{(1,2)\}}$ we have
\[\chi(\pi')=s_3+2qs_{21}+q^2s_{111}.\]
%%where $\pi'$ is the path obtained by inverting the unique corner
%$(2,1)\in c(\pi)$.
By formula \eqref{chi02chi1}, we obtain
\[\chi(\pi,0)=(1-q)^{-1}\left(\chi(\pi)-\chi(\pi')\right)=s_{21}+qs_{111}.\]

%We can also verify Proposition \ref{prop.q1} in this case.
%Only 18 of the 27 terms have no attacks, and 
%the sum on the right hand side of Proposition \ref{prop.q1}
%%will turn out to be 
%\[(1+q)m_{21}+(3+3q)m_{111}.\]
%We may then easily verify that 
%\[\chi(\pi)[X(q-1)]=(q-1)^3 \chi^0(\pi).\]

\end{example}

\begin{example}
We can check that the Dyck path from Example \ref{easydyck} is the 
unique one satisfying $\touch'(\pi)=(1,2)$, and that 
$\bounce(\pi)=1$. 
Therefore, using the calculation that followed we have that 
\[D_{(2,1)}(q,t)=t\chi(\pi,0)=ts_{21}+qts_{111}\]
which can be seen to agree with $\nabla C_1 C_2(1)$. 
\end{example}

\begin{example}

Though we will not need it, this weighted characteristic function
can be used to describe an interesting reformulation
of the formula for the modified Macdonald polynomial given in \cite{haglund2005combinatorial}.
Let $\mu=(\mu_1\geq \mu_2 \geq \cdots \geq \mu_l)$ be a partition of size $n$. Let us list the cells of $\mu$ in the reading order: 
$$
(l,1), (l,2), \ldots, (l, \mu_l), (l-1,1), \ldots, (l-1, \mu_{l-1}) ,\ldots, (1,1),\ldots, (1,\mu_1). 
$$
Denote the $m$-th cell in this list by $(i_m, j_m)$. 

We say that a cell $(i,j)$ attacks all cells which are after $(i,j)$ and before $(i-1,j)$. Thus $(i,j)$ attacks precisely $\mu_i-1$ following cells if $i>1$ and all following cells if $i=1$. Next construct a Dyck path $\pi_\mu$ of length $n$ in such a way that $(m_1, m_2)$ with $m_1<m_2$ is under the path if and only if $(i_{m_1}, j_{m_1})$ attacks $(i_{m_2}, j_{m_2})$. More specifically, the path begins with $\mu_l$ North steps, then it has $\mu_l$ pairs of steps East-North, then $\mu_{l-1}-\mu_l$ North steps followed by $\mu_{l-1}$ East-North pairs and so on until we reach the point $(n-\mu_1, n)$. We complete the path by performing $\mu_1$ East steps. 

Note that the corners of $\pi_\mu$ precisely correspond to the pairs of
cells $(i,j), (i-1, j)$. We set the weight of such a corner to
$q^{\mathrm{arm}(i,j)} t^{-1-\mathrm{leg}(i,j)}$ and denote the weight
function thus obtained by $\mathrm{wt}_\mu$. Note that in our
convention for $\chi(\pi,\mathrm{wt})$ we should count non-inversions in
the corners, while in \cite{haglund2005combinatorial} they count
``descents,'' which translates to counting inversions in the
corners. Taking this into account, we obtain a translation of their
 Theorem 2.2:
$$
\Ht_\mu = q^{-n(\mu')+\binom{\mu_1}{2}} t^{n(\mu)} \chi(\pi_\mu, \mathrm{wt}_\mu). 
$$
\end{example}

%\[d_+ : \DD_k \rightarrow \DD_{k+1},\quad 
%d_- : \DD_{k} \rightarrow \DD_{k-1},\]
%by saying that $d_+$ adds an east step to the 
%starting point of the path,
%raising the upper right 
%end point from $(n,n)$ to $n+1,n+1)$,
%and $d_-$ adds a north step to the beginning of the 
%path. 

%Notice that $d_+$ moves the upper right 
%endpoint from $(n,n)$ to $(n+1,n+1)$.
%%The image of $d_+$ is the subset $\DD'_k\subset \DD_k$ of partial
%Dyck paths whose first step is east. Then any Dyck path $\pi\in \DD$
%of size $n \neq 0$
%may be written uniquely as $\pi=d_-^{k}\pi'$ with
%$\pi' \in \DD'_k$ for some $k$.
%For any such path,
%we denote by $\alpha=\touch'(\pi)$ the composition of $n$ into $k$ parts
%defined by
%\[n-t_i=\mathrm{bounce}(d_{-}^{i+1}d_{+}d_{-}^{l-i} d_{+} \tilde w) - 
%\mathrm{bounce}(w) \quad \text{for $i=0,1,\ldots,l$},\]
%%hence the touch composition $\alpha=(t_1-t_0, \ldots, t_l-t_{l-1})$ can be computed from 
%\[\alpha_i =\mathrm{bounce}(d_{-}^{i}d_{+}d_{-}^{k-i+1}\pi') - 
%\mathrm{bounce}(d_{-}^{i+1}d_{+}d_{-}^{k-i}\pi')\]
%for $1\leq i \leq k$.

%Now let
%%\[D_{\alpha}(q,t)=\sum_{\touch'(\pi)=\alpha} t^{\bounce(\pi)} \chi(\pi,0).\]
%The $(\bounce,\area')$ version of the compositional shuffle conjecture 
%\cite{haglund2012compositional} reads as follows:
%\begin{conjecture} For any composition
%\label{compshuffb} 
%$\alpha=(\alpha_1,...,\alpha_k)$, we have
%\[\nabla C_{\alpha_1}\cdots C_{\alpha_k}(1)=
%%D_{\alpha}(q,t).\]
%\end{conjecture}

\section{Raising and lowering operators}

Now let $\DD_{k,n}$ be the set of Dyck paths from $(0,k)$ to 
$(n,n)$, which we will call partial Dyck paths, and let $\DD_k$ be their union 
over all $n$. For 
$\pi\in\DD_{k,n}$ let $|\pi|=n-k$ denote the number of North steps. 
Unlike $\DD$, the union of the sets $\DD_k$ over all $k$ is closed under the operation of
adding a North or East step to the beginning of the path, and
any Dyck path may be created in such a way starting with the empty path
in $\DD_0$. This is the set of paths that we will develop a recursion for.
More precisely, we will define an extension of the function 
$\chi$ to a map from $\DD_k$ to a new vector space $V_k$, 
and prove that certain operators on these vector spaces commute with
adding North and East steps.

Given a polynomial $P$ depending on variables $u, v$ define
\[(\Delta_{uv} P)(u,v) = \frac{(q-1) v P(u,v) + (v - q u) P(v,u)}{v - u},\]
\[(\Delta_{uv}^* P)(u,v) = \frac{(q-1) u P(u,v) + (v - q u) P(v,u)}{v - u}.\]
We can easily check that $\Delta_{uv}^* = q \Delta_{uv}^{-1}$.
We can recognize these operators as a simple modification of
Demazure-Lusztig operators. The following can be checked by direct computation:
\begin{prop} We have the following relations:
\label{Delprop}
$$
(\Delta_{uv} - q) (\Delta_{uv} + 1) = 0,\qquad (\Delta_{uv}^* - 1) (\Delta_{uv}^* + q) = 0,
$$
$$
\Delta_{uv}\Delta_{vw}\Delta_{uv}=\Delta_{vw}\Delta_{uv}\Delta_{vw},\qquad
\Delta_{uv}^*\Delta_{vw}^*\Delta_{uv}^*=\Delta_{vw}^*\Delta_{uv}^*\Delta_{vw}^*.
$$
%where $u,v,w$ are distinct generators of a polynomial ring.
\end{prop}

\begin{defn}
Let $V_k=\Sym[X]\otimes \Q[y_1,y_2,\ldots,y_{k}]$,
and let 
\[T_i = \Delta_{y_i y_{i+1}}^* : V_{k} \rightarrow
V_{k},\quad i=1,\ldots,k-1.\]
% for $i=1,\ldots,k-1$. 
Define operators
\[d_{+} : V_k \rightarrow V_{k+1},
\quad d_- : V_{k} \rightarrow V_{k-1}\]
by
%Let $d_{+}:V_k \rightarrow V_{k+1}$ be the operator which sends $F$ to 
\begin{equation}
\label{dpdef}
(d_{+} F)[X] = T_1 T_2\cdots T_k \left(F[X+(q-1)y_{k+1}]\right),
\end{equation}
%Then for any $w\in\DD_k$
%
%Let $d_{-}:V_k\rightarrow V_{k-1}$ be the operator which send
%for $F\in V_k$,
% where the plethystic operators act only on the $\Sym[X]$ part,
 and
\begin{equation}
\label{dmdef}
(d_- F)[X] = -F[X-(q-1)y_k] \pExp[-y_k^{-1} X] |_{y_k^{-1}}
\end{equation}
for $F\in V_k$.
%\begin{equation}
%\label{dmdef}
%d_-(y_k^i F)= -B_{i+1} F
%\end{equation}
%when $F$ does not depend on $y_k$. 
%Then for any $w\in\DD_k$
\end{defn}
\begin{rem}
Note that the operator $d_-$ is related to the $B_i$ operators as follows:
\[
d_-(y_k^i F)= -B_{i+1} F
\]
for $F\in V_k$ which do not depend on $y_k$. 
\end{rem}

We now claim the following theorem:
\begin{thm}
\label{dpthm}
For any Dyck path $\pi$ of size $n$, 
let $\eps_1\cdots \eps_{2n}$ denote the 
corresponding sequence
of plus and minus symbols where a plus denotes an
east step, and a minus denotes a north step reading
$\pi$ from bottom left to top right. Then
\[\chi(\pi)=d_{\eps_1}\cdots d_{\eps_{2n}}(1)\]
as an element of $V_0=\Sym[X]$.
\end{thm}

\begin{example}
Let $\pi$ be the Dyck path from Example \ref{easydyck}.
We have that
\[d_- d_- d_+d_+d_-d_+(1)=
d_- d_- d_+d_+d_-(1)=d_- d_- d_+d_+(s_1)=\]
\[d_- d_- d_+\left(s_1+(q-1)y_1\right)=
d_- d_- \left(s_1+(q-1)(y_1+y_2)\right)=\]
\[d_- \left(s_2+s_{11}+(q-1)s_1y_1\right)=s_3+(1+q)s_{21}+qs_{111},\]
which agrees with the value calculated for $\chi(\pi)$.
\end{example}

Combining this result with equation \eqref{chi02chi1} implies the
following:
\begin{cor}
\label{chi0cor}
The following procedure computes $\chi(\pi,0)$: start with $1\in \Sym[X]=V_0$, follow the path from right to left applying $\frac{1}{q-1}[d_{-}, d_{+}]$ for each corner of $w$, and $d_{-}$ ($d_{+}$) for each North (resp. East) step which is not a side of a corner. 
\end{cor}

\subsection{Rank experiment}
The proof of Theorem \ref{dpthm} will be divided in several parts.
However, before we proceed to the proof of Theorem \ref{dpthm}, we would like to
explain why we expected such a result to hold and how we obtained it. 
In fact, the definition of $\chi_k$ from equation \eqref{chikdef} in the proof below actually came first,
% the definitions of $d_{\pm}$,
and was discovered using computer experimentation, as we now explain.

First note that the number of Dyck paths of length $n$ is given by the Catalan number $C_n=\frac{1}{n+1} \binom{2n}{n}$ which grows exponentially with $n$. The dimension of the degree $n$ part of $\Sym[X]$ is the number of partitions of size $n$, which grows subexponentially. For instance, for $n=3$ we have $5$ Dyck paths, but only $3$ partitions. Thus there must be linear dependences between different $\chi(\pi)$. 

Now fix a partial Dyck path $\pi_1\in \DD_{k,n}$. For each partial Dyck path $\pi_2\in\DD_{k,n'}$ we can reflect $\pi_2$ and concatenate it with $\pi_1$ to obtain a full Dyck path $\pi_2^{op}\pi_1$ of length $n+n'-k$. We may then consider its character $\chi(\pi_2^{op}\pi_1)$. We keep $n$, $\pi_1$ fixed and vary $n'$, $\pi_2$, thus obtaining a map $\varphi_{\pi_1}:\DD_k\to\Sym[X]$. The map $\pi_1\to \varphi_{\pi_1}$ is a map from $\DD_k$ to the vector space of maps from $\DD_k$ to $\Sym[X]$, which is very high dimensional, because both the set $\DD_k$ is infinite and $\Sym[X]$ is infinite dimensional. A priori, it could be the case that the images of the elements of $\DD_{k,n}$ in $\mathrm{Maps}(\DD_k,\Sym[X])$ are linearly independent. However, computer experiments convinced us that it is not the case, and that there should be a vector space $V_{k,n}$ whose dimension is generally smaller than the size of $\DD_{k,n}$.
In fact, by restricting $n'$ to be bounded by some arbitrary but large enough cutoff value, 
we were able to predict that the dimension of this space stabilizes to a very simple formula, which is 
the dimension of $V_{k,n}$, the degree $n-k$ component of $V_{k}$ as it is defined above.

We therefore predicted the existence of a commutative diagram
\[
\begin{tikzcd}
\DD_{k,n} \arrow{r}{\chi_{k,n}} \arrow{d}{\varphi} & V_{k,n}\arrow{ld}\\
\mathrm{Maps}(\DD_k,\Sym[X]) & 
\end{tikzcd}
\]
for some map $\chi_{k,n}$, whose image spans all of $V_{k,n}$.
This ultimately led to the guess of the formula for $\chi_k$ in 
\eqref{chikdef} as the correct extension of $\chi(\pi,1)$.
It is not at all trivial to deduce this formula 
from the dimension of $V_{k,n}$, and indeed, some substantial
guesswork was required. However, the validity of any particular
guess $\chi_{k,n}$ can be determined experimentally, by testing
if its kernel in the $\C(q)$-span of $\DD_{k,n}$ agrees with the kernel
of $\varphi$. Clearly the existence of a testable criterion 
such as this makes the problem of determining
%the reader  makes the plausibility
$\chi_{k,n}$ experimentally much more reasonable.

Once the definition of $\chi_{k,n}$ was conjectured, finding formulas for $d_{\pm}$
that satisfy \eqref{eq:recs1}
turned out to be relatively straightforward.

%the following observation should
%convince 
%to this
%To guess the actual formula

%We then guessed that $V_{k,n}$ should be the degree $n$ part of $V_k=\Sym[X]\otimes\Q[y_1,\ldots,y_k]$, and from that conjectured a definition of $\chi_k:\DD_k\to V_k$ defined below, and verified on examples that partial Dyck paths that are linearly dependent after applying $\chi_{k,n}$ satisfy the same linear dependence after applying $\varphi$. Once this was established, the computation of the operators $d_{-}$, $d_{+}$ and proof of Theorem \ref{dpthm} turned out to be relatively straightforward.

%\begin{rem}\label{rem.formal}
%The space $V_k$ is spanned by elements of the form $y_1^{i_1} y_2^{i_2} \cdots y_k^{i_k} F$ for $F\in\Sym[X]$. Sometimes it is useful to think of such an element as a formal composition of operators applied to $F$, written as
%$$
%(-1)^k B_{i_1+1}\ast B_{i_2+1}\ast\cdots \ast B_{i_k+1} \ast F.
%$$
%It is formal in the following sense: instead of simply evaluating this expression to an element of $\Sym[X]$, we keep track of the sequence of the operators. Then $d_{-}$ simply applies the rightmost operator to $F$. However, it is more complicated to write the operator $d_{+}$ using this notation.
%\end{rem}

\subsection{Characteristic functions of partial Dyck paths}

The following definition is motivated by Proposition \ref{prop.q1}. Let $\pi \in \DD_{k,n}$. Let $\sigma=(\sigma_1, \sigma_2,\ldots,\sigma_k)\in \Z_{>0}$ be a tuple of distinct numbers. The elements of $\Im(\sigma)\subset \Z_{>0}$ will be called \emph{special}. Let
\[
U_{\pi,\sigma}=\left\{w\in \mathbb{Z}^{n}_{>0} :
w_i=\sigma_i \mbox{ for } i \leq k,\ 
w_i \neq w_j \mbox{ for } (i,j) \in \Area(\pi)\right\}.
\]
The second condition on $w$ is the ``no attack'' condition as before. The first condition says that we put the special labels in the positions $1,2,\ldots,k$ as prescribed by $\sigma$. Let 
\begin{equation}\label{eq:chiprime}
\chi_\sigma'(\pi)=\sum_{w \in U_{\pi,\sigma}} q^{\inv(\pi,w)} z_w.
\end{equation}
Here we use variables $z_1,z_2,\ldots$.

Suppose $\sigma$ is a permutation, i.e $\sigma_i\leq k$ for all $i$. Set $z_i=y_i$ for $i\leq k$ and $z_i=x_{i-k}$ for $i>k$.
We denote
\[
\chi'_{k}(\pi)=\chi_{(1,2,\ldots,k)}'(\pi).
\] 

Let us group the summands in \eqref{eq:chiprime} according to the positions of
special labels. More precisely, let $S\subset\{1,\ldots,n\}$ such that $\{1,\ldots,k\}\subset S$ and $w^S:S\to \{1,\ldots,k\}$ such that $w^S_i=\sigma_i$ for $i=1,2,\ldots,k$ and $w_i\neq w_j$ for $i,j\in S$, $(i,j)\in\Area(\pi)$. Set 
\[
U_{\pi,\sigma}^{S,w^S}:=\left\{w\in U_{\pi,\sigma} : 
w_i=w^S_i\mbox{ for } i \in S,\; w_i>k\mbox{ for } i \notin S\right\},
\]
\[
\Sigma_{\pi,\sigma}^{S,w^S}:=\sum_{w \in U_{\pi,\sigma}^{S,w^S}} q^{\inv(\pi,w)} x_w,\qquad \chi_\sigma'(\pi) = \sum_{S, w^S} \Sigma_{\pi,\sigma}^{S,w^S}.
\]
Let $m_1<m_2<\cdots<m_r$ be all the positions not in $S$.
Let $\pi_S$ be the unique Dyck path of length $r$ such that $(i,j)\in\Area(\pi_S)$ if and only if $m_i, m_j\in \Area(\pi)$. We have
\[
\Sigma_{\pi,\sigma}^{S,w^S} = q^A \prod_{i\in S} y_{w_i} \sum_{w \in \Z_{>0}^{r} \text{no attack}} q^{\inv(\pi_S, w)} x_w,
\]
where 
\[
A=\#\{(i,j)\in\Area(\pi) : (i\in S, j\in S, w^S_i>w^S_j) \mbox{ or } (i\notin S, j\in S)\}.
\]

By Proposition \ref{prop.q1} we have
\begin{equation}\label{eq:chitochik}
\Sigma_{\pi,\sigma}^{S,w^S} = q^A (q-1)^{|S|-n} \chi(\pi^S)\left[(q-1)X\right] \prod_{i\in S} y_{w_i}.
\end{equation}
In particular, $\chi_\sigma(\pi)$ is a symmetric function in $x_1, x_2,\ldots$ and it makes sense to define
\begin{equation}
\label{chikdef}
\chi_\sigma(\pi)[X] := \frac{1}{y_1 y_2\cdots y_k} (q-1)^{|\pi|}\chi_\sigma'(\pi)\left[\frac{X}{q-1}\right]\in V_k,\quad \chi_k(\pi):=\chi_{\Id_k}(\pi).
\end{equation}

\begin{rem}
The identity (\ref{eq:chitochik}) also implies that the coefficients of $\chi_\sigma(\pi)[X]$ are polynomials in $q$ and gives a way to express $\chi_\sigma$ in terms of the characteristic functions $\chi(\pi_S)$ for all $S$.
\end{rem}

For $k=0$ we recover $\chi(\pi)$:
\[
\chi_0(\pi)=\chi(\pi)\qquad (\pi\in \mathbb{D}_0=\mathbb{D}).
\]
Thus, it suffices to prove that
\begin{equation}
\label{eq:recs1}
\chi_{k+1}(E\pi)=d_{+} \chi_{k}(\pi),\quad
\chi_{k-1}(N\pi)=d_{-} \chi_{k}(\pi)
\quad (\pi \in \DD_k).
\end{equation}

%For instance, if $\pi$ is the Dyck path from Example \ref{easydyck},
%we can see that
%\[\chi'_0(\pi)=(1+q)s_{21}+(1+q)s_{111}.\]
%inserting this into \eqref{chikdef} produces
%the value of $\chi(\pi)$ computed above.
%which by Proposition \ref{} agrees with $\chi(\pi)=\chi(\pi)$ for $k=0$.
%The operation $F[X] \rightarrow F[(q-1)X]$ is a bijection on symmetric functions. 
%Therefore $\chi_\sigma(\pi)$ 
%is well-defined, and it follows from the proposition that
%$\chi_0(\pi)=\chi(\pi,1)$.
%\[\chi_k(\pi) = \chi_{1_k}(\pi),\quad 1_k=1\in S_k\]
%by adding a north or east step to the path
%respectively.
%Notice that $d_+$ 
\subsection{Raising operator}
We begin with the first case. Let $\pi\in\DD_{k,n}$ so that $E\pi \in \DD_{k+1, n+1}$, and we need to express $\chi_{k+1}(E\pi)$ in terms of $\chi_k(\pi)$. Let $\sigma$ be the following sequence:
\[
\sigma=(k+1, 1, 2, \ldots, k).
\]
Then we have a bijection $f:U_{\pi,\Id_k} \to U_{E\pi,\sigma}$ obtained
by sending 
\[
w=(1, 2, \ldots, k, w_{k+1},\ldots,w_n)
\]
to
\[
f(w):=(k+1, 1, 2, \ldots,k, w_{k+1},\ldots,w_n).
\]
This is possible because $1$ does not attack $k+1$ in $E\pi$.
We clearly have $\inv({E\pi}, f(w)) = \inv(\pi, w) + k$, which implies
\begin{equation*}
\chi'_{\sigma}(E\pi) = z_{k+1} q^k \chi'_k(\pi),
\end{equation*}
where both sides are written in terms of the variables $z_i$. When we pass to the variables $x_i$, $y_i$ on the left we have
\[
(z_1,z_2,\ldots)=(y_1,y_2,\ldots,y_{k+1},x_1,x_2,\ldots),
\]
but on the right we have
\[
(z_1,z_2,\ldots)=(y_1,y_2,\ldots,y_{k},x_1,x_2,\ldots),
\]
thus we need to perform the substitution $X=y_{k+1}+X$:
\[
\chi'_{\sigma}(E\pi)[X] = y_{k+1} q^k \chi'_k(\pi)[X+y_{k+1}],
\]
Performing the transformation \eqref{chikdef} we obtain
\begin{equation}
\label{chi1dp}
\chi_{\sigma}(E\pi) = q^k \chi_k(\pi)\left[X+(q-1)y_{k+1}\right].
\end{equation}
To finish the computation we need to relate $\chi_{k+1}=\chi_{\Id_{k+1}}$ and $\chi_\sigma$. We first note that $\sigma$ can be obtained from $\Id_{k+1}$ by successively swapping neighboring labels. Let $\sigma^{(1)}=\Id_{k+1}$ and
\[
\sigma^{(i)} = (i, 1, 2,\ldots,i-1,i+1,\ldots,k+1)\qquad (i=2,3,\ldots,k+1),
\]
so that $\sigma = \sigma^{(k+1)}$. It is clear that $\sigma^{(i+1)}$ can be obtained from $\sigma^{(i)}$ by interchanging the labels $i$ and $i+1$.

We show below (Proposition \ref{prop:swapping}) that this kind of interchange is controlled by the operator $\Delta_{y_i, y_{i+1}}$:
\begin{equation}
\label{chisigrec}
\chi_{\sigma^{(i+1)}}(E\pi)=\Delta_{y_i,y_{i+1}}
\chi_{\sigma^{(i)}}(E\pi).
\end{equation}

This implies
\[
\chi_\sigma(E\pi) = \Delta_{y_{k-1},y_{k}}\cdots \Delta_{y_{1},y_{2}} \chi_{k+1}(E\pi).
\]
When we insert this equation
into \eqref{chi1dp}, we arrive at
\[
\chi_{k+1}(E\pi) = T_{1}\cdots T_{k}
\left(\chi_k(\pi)\left[X+(q-1) y_{k+1}\right]\right) = d_+ \chi_k(\pi).
\]

\subsection{Swapping operators}
\begin{prop}\label{prop:swapping}
For any $w\in \DD_k$, $\sigma$ as above and $m$ special suppose that
$m+1$ is not special or $\sigma^{-1}(m)<\sigma^{-1}(m+1)$. Then we have
\[
\chi'_{\tau_m \sigma}(w) = \Delta_{z_{m}, z_{m+1}} \chi'_{\sigma},
\]
where $\tau_m$ is the transposition $m\leftrightarrow m+1$, $(\tau_m\sigma)_i = \tau_m(\sigma_i)$ for $i=1,\ldots,k$.
\end{prop}
\begin{proof}
We decompose both sides as follows. For any  $w \in U_{\pi, \sigma}$ let $S(w)$ be the set of indices $j$ where $w_j\in\{m,m+1\}$. For $w,w'\in U_{\pi, \sigma}$ write $w\sim w'$ if $S(w)=S(w')$ and $w_i=w_i'$ for all $i\notin S(w)$. This defines an equivalence relation on $U_{\pi, \sigma}$.
The sum \eqref{eq:chiprime} is then decomposed as follows:
\begin{equation}\label{eq:rundecomp}
\chi'_\sigma(\pi) = \sum_{[w]\in U_{\pi, \sigma}/\sim} q^{\inv_1(\pi, w)} \prod_{i\notin S} z_{w_i} \sum_{w'\sim w} a(w'),
\end{equation}
where 
\[
\inv_1(\pi, w) = \#\{(i,j)\in\Area(\pi): w_i>w_j, i\notin S(w) \mbox{ or } j\notin S(w)\},
\]
which does not depend on the choice of a representative $w$ in the equivalence class $[w]$, and 
\[
a(w) = q^{\inv_2(\pi, w)} \prod_{i\in S} z_{w_i},
\]
\[
\inv_2(\pi, w) = \#\{(i,j)\in\Area(\pi): w_i>w_j, i,j \in S(w)\}.
\]

Let $f:U_{\pi, \sigma}\to U_{\pi, \tau_m\sigma}$ be the bijection defined by $f(w)_i = \tau_m(w_i)$. This bijection respects the equivalence relation $\sim$ and we have $S(f(w))=S(w)$. Moreover, we have $\inv_1(\pi, w) = \inv_1(\pi, f(w))$. 
We now make the stronger claim that for any $w\in U_{\pi, \sigma}$
\begin{equation}
\label{chisigrecref}
\sum_{w'\sim f(w)} a(w') = \Delta_{z_m,z_{m+1}}
\sum_{w'\sim w} a(w')
\end{equation}
which would imply the statement by summing over all equivalence classes. 

For each $w\in U_{\pi, \sigma}$ the set $S(w)$ is decomposed into a disjoint union of \emph{runs}, i.e. subsets
\[R=\{j_1,...,j_l\} \subset \{1,...,n\},\quad j_1<\cdots <j_l\]
such that in each run $j_a$ attacks $j_{a+1}$ for all $a$ and
elements of different runs do not attack each other. 
Because of the non attacking condition, the labels $w_{j_a}$ must alternate between $m,m+1$ and $j_a$ does not attack $j_{a+2}$. Thus to fix $w$ in each equivalence class it is enough to fix $w_{j_1}$ for each run. Suppose the runs of $S(w)$ have lengths $l_1, l_2,\ldots, l_r$ and the first values of $w$ in each run are $c_1, c_2,\ldots, c_r$ respectively.

With this information $a(w)$ can be computed as follows:
\[
a(w) = \prod_{i=1}^r a(l_i, c_i),
\]
where
\[
a(l,c):=
\begin{cases}
q^{l'-1} z_{m}^{l'} z_{m+1}^{l'} & l=2l', c=m\\
q^{l'} z_{m}^{l'+1} z_{m+1}^{l'} & l=2l'+1, c=m \\
q^{l'} z_{m}^{l'} z_{m+1}^{l'} & l=2l', c=m+1 \\
q^{l'} z_{m}^{l'} z_{m+1}^{l'+1} & l=2l'+1, c=m+1.
\end{cases}
\]

For instance, let $k=3$ and $\pi$ be the Dyck path in Figure \ref{dyckproof},
and let 
\[w=(1,3,2,7,1,7,1,2) \in U_{\pi,(132)}.\]
Let $m=1$. Then we have $S(w)=\{1,3,5,7,8\}$, which decomposes into
two runs
$\{1,3,5\}$ and $\{7,8\}$. So we have $r=2$, $(l_1, l_2)=(3, 2)$, $(c_1, c_2) = (1, 1)$ and we obtain
\[
a(w)=a(3, 1) a(2, 1) = q z_1^2 z_2 z_1 z_2 = q z_1^3 z_2^2.
\]
\begin{figure}
\begin{tikzpicture}
\draw[help lines] (0,0) grid (8,8);
\draw[dashed, color=gray] (0,0)--(8,8);
\draw[->,very thick] (0,3)--(0,4);
\draw[->,very thick] (0,4)--(1,4);
\draw[->,very thick] (1,4)--(2,4);
\draw[->,very thick] (2,4)--(2,5);
\draw[->,very thick] (2,5)--(3,5);
\draw[->,very thick] (3,5)--(4,5);
\draw[->,very thick] (4,5)--(4,6);
\draw[->,very thick] (4,6)--(5,6);
\draw[->,very thick] (5,6)--(5,7);
\draw[->,very thick] (5,7)--(5,8);
\draw[->,very thick] (5,8)--(6,8);
\draw[->,very thick] (6,8)--(7,8);
\draw[->,very thick] (7,8)--(8,8);
\node at (.7,.3) {1};
\node at (1.7,1.3) {3};
\node at (2.7,2.3) {2};
\node at (3.7,3.3) {7};
\node at (4.7,4.3) {1};
\node at (5.7,5.3) {7};
\node at (6.7,6.3) {1};
\node at (7.7,7.3) {2};
\end{tikzpicture}
\caption{}
\label{dyckproof}
\end{figure}

 Note that by the assumption on $\sigma$ we have $c_1=m$, while $c_i$ can take arbitrary values $\{m, m+1\}$ for $i>1$. This implies
\[
\sum_{w'\sim w} a(w') = a(l_1, m) \prod_{i=2}^r (a(l_i, m) + a(l_i, m+1)).
\]
On the other hand we have
\[
\sum_{w'\sim f(w)} a(w') = \sum_{w'\sim w} a(f(w')) =  a(l_1, m+1) \prod_{i=2}^r (a(l_i, m) + a(l_i, m+1)).
\]

Now notice that for all $l$ the sum $a(l, m) + a(l, m+1)$ is symmetric in $z_m, z_{m+1}$. The operator $\Delta_{z_m,z_{m+1}}$ commutes with multiplication by symmetric functions and satisfies 
\[\Delta_{z_m,z_{m+1}}(a(l, m)) = a(l, m+1).\]
This establishes \eqref{chisigrecref} and the proof is complete.
\end{proof}
\begin{rem}
The arguments used in the proof can be used to show that in the case when $m, m+1$ are both not special the function $\chi'_\sigma(\pi)$ is symmetric in $z_m, z_{m+1}$. In particular, we can obtain a direct proof of the fact that $\chi'_\sigma$ is symmetric in the variables $z_m, z_{m+1}, z_{m+2},\ldots$ for $i=\max(\sigma)+1$, without use of Proposition \ref{prop.q1}.
\end{rem}

\subsection{Lowering operator}

We now turn to the remaining identity $\chi_{k-1}(N\pi)=d_-\chi_k(\pi)$. Assume $\pi\in\DD_{k,n}$, so that $N\pi\in\DD_{k-1,n}$. We observe that
\[\chi'_{k-1}(N\pi)[X + y_{k}] =
\sum_{r\geq 0} \chi'_{k,r}(\pi)[X],\]
where
\[\chi'_{k,r}(\pi)=\chi'_{\sigma}(\pi),\quad \sigma=(1,2,...,k-1,k+r)\]
and to get to the second equality we have summed over all possible values 
of $r=w_k-k$ that do not result in an attack. It is convenient to set $x_0=y_k$.
Using Proposition \ref{prop:swapping} we can characterize $\chi'_{k,r}(\pi)$ by
\begin{equation}
\label{chi1r}
\chi'_{k,0}(\pi) = \chi'_{k}(\pi),\quad
\chi'_{k,r+1}(\pi) = \Delta_{x_r,x_{r+1}} \chi'_{k,r}(\pi)\quad(r\geq 0).
\end{equation}

Now notice that there is a unique expansion
\[\chi'_{k}(\pi)[X]=\sum_{j \geq 1} y_k^j g_j(\pi)[X+y_k],\quad g_j(\pi)\in V_{k-1}.\]
The advantage over the more obvious expansion in powers of $y_k$ is that each
coefficient $g_j[X+y_k]$ is symmetric in the variables $y_k,x_1,...$
As a result, we have that
\[\chi'_{k,r}(\pi)[X]=\Delta_{x_{r-1},x_{r}}\cdots \Delta_{x_2,x_1}\Delta_{y_k,x_1} \sum_{i\geq 1} y_k^i g_i(\pi)[X+y_k]=
\sum_{i\geq 1} f_{i,r} g_i(\pi)[X+y_k]\]
where
\[f_{i,r}=\Delta_{x_{r-1},x_{r}}\cdots 
\Delta_{x_1,x_2}\Delta_{y_k,x_1}(y_k^i)\quad(i\geq 1, r\geq 0)
\]
The extra symmetry in the $y_k$ variable is used to
pass $\Delta_{y_k,x_1}$ by multiplication by $g_i(\pi)[X+y_k]$.

Now we need an explicit formula for $f_{i,r}$:
\begin{prop}
Denote $X_r=y_k+x_1+\cdots+x_r$, $X_{-1}=0$. We have
\[
f_{i,r}=\frac{h_i[(1-q)X_r] - h_i[(1-q)X_{r-1}]}{1-q}.
\]
\end{prop}
\begin{proof}
Denote the right hand side by $f_{i,r}'$. The proof goes by induction on $r$. For $r=0$ both sides are equal to $y_k^i$. Thus it is enough to show
that
\begin{equation}
\label{fjrrec}
\Delta_{x_{r},x_{r+1}}(f'_{i,r})=f'_{i,r+1}.
\end{equation}
Use $X_r=X_{r-1}+x_r$ to write $f_{i,r}'$ as follows:
\begin{equation}
\label{fjreq}
f_{i,r}' = \sum_{j=1}^i x_r^j h_{i-j}[(1-q) X_{r-1}] = x_r h_{i-1}[(1-q) X_{r-1} + x_r].
\end{equation}
Now $X_{r-1}$ does not contain the variables $x_r$, $x_{r+1}$, so we have
\[
\Delta_{x_{r},x_{r+1}}(f'_{i,r}) = \sum_{j=1}^i h_{i-j}[(1-q) X_{r-1}]  \Delta_{x_{r},x_{r+1}} x_r^j.
\]
Using the formula
\[\Delta_{x_r, x_{r+1}} x_r^j = 
x_{r+1} h_{j-1}[(1-q)x_r + x_{r+1}],\]
which is straightforward to check, we can evaluate
\[
\Delta_{x_r x_{r+1}} f'_{i,r} = 
x_{r+1}\sum_{j=1}^i h_{j-1}[(1-q)x_r + x_{r+1}] h_{i-j}[(1-q)X_{r-1}]
\]
\[
=
x_{r+1} h_{i-1}[(1-q) X_{r} + x_{r+1}],
\]
which matches $f_{i,r+1}'$ by \eqref{fjreq}.
\end{proof}

Now, if we sum over all $r$, we obtain
\begin{equation}
\label{sumfir}
\sum_{r \geq 0} f_{i,r}=
(1-q)^{-1}h_i\left[(1-q)(X+y_k)\right].
\end{equation}
Thus
\[\chi_{k-1}'(N \pi)[X + y_{k}] = 
(1-q)^{-1} \sum_{i\geq 1} h_i[(1-q)(X+y_k)] g_i(\pi)[X+y_k].\]
This implies
\begin{equation}
\label{chistep}
\chi_{k-1} (N\pi)[X] = - \frac{(q-1)^{n-k}}{y_1\cdots y_{k-1}} \sum_{i\geq 1} h_i[-X] g_i(\pi)\left[\frac{X}{q-1}\right].
\end{equation}
On the other hand $g_i(\pi)$ were defined in such a way that
\[
\chi_k(\pi)[(q-1)X] = \frac{(q-1)^{n-k}}{y_1\cdots y_{k}} \sum_{i\geq 1} y_k^i g_i(\pi)[X+y_{k}].
\]
Substituting $\frac{1}{q-1}X-y_k$ for $X$ gives
\begin{equation}
\label{chistep2}
\chi_k(\pi)[X-(q-1)y_k] = \frac{(q-1)^{n-k}}{y_1\cdots y_{k}} \sum_{i\geq 1}  y_k^i g_i(\pi)\left[\frac{X}{q-1}\right].
\end{equation}
Comparing \eqref{chistep} and \eqref{chistep2} we obtain
\[
\chi_{k-1} (N\pi)[X] = \sum_{i\geq 0} -h_{i+1}[-X]\left(\chi_k(\pi)[X-(q-1)y_k]|_{y_k^i}\right).
\]
This can be seen to coincide with $d_- \chi_k(\pi)$, establishing the second case of \eqref{eq:recs1}. Thus the proof of Theorem \ref{dpthm} is complete.

\subsection{Main recursion}
We now show how to express all of $D_{\alpha}(q,t)$ using our operators:
\begin{thm}
\label{thm:recN}
If $\alpha$ is a composition of length $l$, we have
\[D_{\alpha}(q,t)=d_-^l(N_\alpha).\]
where $N_\alpha\in V_l$ is defined by the recursion relations
\begin{equation}
\label{receqs}
N_{\emptyset}=1,\quad N_{[1,\alpha]}=d_+N_{\alpha},\quad 
N_{a \alpha} = \frac{t^{a-1}}{q-1} [d_{-},d_{+}] \sum_{\beta\models
  a-1} d_{-}^{l(\beta)-1} N_{\alpha \beta}.
\end{equation}
\end{thm}

\begin{proof}
For any $k> 0$ let $\DD_k^0\subset \DD_k$ denote the subset of partial Dyck paths that begin with an East step. For $k=0$ let $\DD_0^0=\{\emptyset\}$. 
Define functions $\chi^0:\DD_k^0\rightarrow V_k$ by
\[\chi^0(\emptyset)=1,\quad \chi^0(EN^i\pi)=\frac{1}{q-1}[d_{-},
d_{+}] d_-^{i-1}\chi^0(\pi),\]
\[\chi^0(E\pi)=d_+\chi^0(\pi).\]
Given a composition $\alpha$ of length $l$, let
\[
\DD_{\alpha}=\left\{\pi \in \DD : \touch'(\pi)=\alpha\right\}.
\]
By the definition of $\touch'$ every element of $\DD_{\alpha}$ is of the form $\pi=N^l \tilde{\pi}$ for a unique element $\tilde{\pi}\in \DD_l^0$ so that by Corollary \ref{chi0cor} we have
\[
\chi(\pi, 0) = d_-^l \chi^0_l(\tilde{\pi}).
\]
Let
\[
N'_\alpha = \sum_{\pi \in \DD_\alpha} 
t^{\bounce(\pi)}\chi_l^0(\tilde{\pi}) \in V_l,
\]
so that $D_{\alpha}(q,t)=d_-^l(N'_{\alpha})$.
It suffices to show that $N'_\alpha$ satisfies the relations
\eqref{receqs}, and so agrees with $N_\alpha$.

For a composition $\alpha$ of length $l$ and $0\leq r\leq l$ we have a map
$\gamma_{\alpha, r}: \DD_\alpha \to \DD$
as follows: $\gamma_{\alpha, r}(\pi) = N^{r+1} E N^{l-r} \tilde{\pi}$. Clearly $|\gamma_{\alpha, r}(\pi)| = |\pi|+1$. From the definition of $\touch'$ we see the following relation:
\[
\bounce(\gamma_{\alpha, r}(\pi)) = \bounce(\pi) + \sum_{i>r} \alpha_i.
\]
Next we compute $\touch'(\gamma_{\alpha, r}(\pi))$. For $0\leq i\leq r$ we have
\[
\bounce(N^{i+2} E N^{r-i} E N^{l-r} \tilde{\pi}) = \bounce(N^{i+1} E N^{l-i} \tilde{\pi}).
\]
This implies
\[
\touch'(\gamma_{\alpha, r}(\pi)) = \left(1+\sum_{i>r} \alpha_i,\, \alpha_1,\, \alpha_2,\ldots,\, \alpha_r\right),
\]
so in particular $\touch'(\gamma_{\alpha, r}(\pi))$ depends only on $\alpha$ and $r$.

Since every non-empty Dyck path can be obtained as $\gamma_{\alpha, r}(\pi)$ in a unique way, we obtain for every composition $\alpha$ of length $r$:
\[
\DD_{a\alpha} = \bigsqcup_{\beta\models a-1} \gamma_{\alpha\beta, r}(\DD_{\alpha\beta}).
\]
It is not hard to see that this identity precisely translates to the relations \eqref{receqs} for $N_\alpha'$.
\end{proof}

\begin{example}
Using Theorem \ref{thm:recN}, we find that 
\[N_{31}=\frac{t^3}{(q-1)^2}
\left(d_{-++-++}-d_{-+++-+}-d_{+-+-++}+d_{+-++-+}\right)+\]
\[\frac{t^2}{q-1}\left(d_{-+-+++}-d_{+--+++}\right)=
qt^3y_1^2-qt^2y_1 e_1\in V_2,\]
where $d_{\epsilon_1\cdots \epsilon_n}=d_{\epsilon_1}\cdots d_{\epsilon_n} (1).$
We may then check that 
\[d_-^2 N_{31}=qt^3B_{3}B_1(1)+qt^2B_{2}B_1B_1(1)=\nabla C_{3}C_1(1).\]
\end{example}

\begin{example}
%It is clearly an important question to determine the meaning behind $N_\alpha$.
%To answer this question completely will require a geometric understanding of
%$V_k$, which is the subject of an upcoming paper by E. Gorsky and the authors,
%and will not be explained here. However, in the following example,
%we will connect these quantities with the recursive formulas from previous papers about the
%compositional case of the $q,t$ Catalan numbers \cite{}.

Let $\alpha=(a_1,...,a_k)$ be a composition of norm $n$ and length $k$ and
consider the polynomials in $q,t$ given by
\[D^{(-)}_{\alpha}(q,t)=\langle D_{\alpha}(q,t), e_n \rangle=(-1)^n D_{\alpha}(q,t)[-1],\]
which encode the Catalan case of the compositional shuffle conjecture.
This case of the conjecture was proved in
\cite{GXZ2012}, using recursions (Proposition 3.12) that are similar 
those defining $N_{\alpha}$ \eqref{receqs}, and it
is natural to ask if they follow as a special case.
%We claim that these are the same polynomials defined in \cite{},
%and that the recursions of Theorem ? follow from \eqref{}.
It is not true that \eqref{receqs} descends to a recursion 
for the coefficients in from of $s_{1^n}$. However, we have the following proposition:

\begin{prop}
\label{catrecprop}
Let $\pi$ be Dyck path of size $N$ ending with $k$ East steps,
and let $(\epsilon_1,...,\epsilon_{2N-k}, +^k)$ denote the corresponding
sequence of plus and minus symbols, as in Theorem \ref{dpthm}.
Then for $f\in V_k$, we have
\begin{equation}
\label{catreceq}
\left(d_{\eps_1}\cdots d_{\eps_{2N-k}} (f)\right)[-1]=
(-1)^N q^{\area(\pi)} f[-q^k]\big|_{y_i=q^{i-1}}.
\end{equation}
%\langle d_-^k d_+(s_{\mu}y^{\alpha_1}\cdots y^{\alpha_k}),e_d\rangle =
%\begin{cases}
%(-1)^{N} q^{\area(\pi)+kn+\sum_{i=1}^k (i-1) \alpha_i} & \mu=1^n\\
%0 & \mbox{otherwise} \\
%\end{cases}
%where $d$ is the degree of the expression.
%\end{equation}
%where $f\in V_k'$, the composition of $d_{\pm}$ operators take $V_k$ to $V_0$,
%and $d$ is the degree of the polynomial on the left.
\end{prop}

\begin{proof}
We prove this by induction on $2N-k$, the number of $d_\pm$ symbols.

First, suppose the final step $\epsilon_{2N-k}$ is a plus. 
Let $\rho_k$ denote the homomorphism $f\mapsto f[-q^k]|_{y_i=q^{i-1}}$
on the right hand side of \eqref{catreceq}.
Using the definition \eqref{dpdef} of $d_+$, and the observation that $\rho_k T_i=\rho_k$,
we have that
\[\rho_{k+1}(d_+f)=
\rho_{k+1}(f[X+(q-1)y_{k+1}])=
\rho_k(f).\]
The case now follows from the above equation, the induction step,
and the fact that the path $\pi$ remains the same.
%we have that the left hand side of \eqref{catreceq} is given by

The second case follows easily from the formula
\[\rho_{k}(d_-(fy_{k+1}^a))=-q^{d+1-(k-1)a} \rho_k(f),\]
for $f\in V_k$ homogeneous of degree $d$ in the symmetric function part. 
The base case is obvious.

\end{proof}

%In Proposition 3.12 of \cite{}, the authors give a recursive formula
%formula for $D_{\alpha}^{(-)}(q,t)$ (an an extension for a general hook)
%that is similar looking to
%the relations \eqref{receqs} for $N_{\alpha}$. 
The recursions of \cite{GXZ2012} now follow fairly easily
from \eqref{receqs} as a special case,
%shows that they do indeed follow as a special case, 
%if
if Theorem \ref{thm:recN} is taken as the definition of $D_{\alpha}(q,t)$.
The relations for $N_\alpha$ were not discovered this way, and 
the authors only noticed this connection after following up on 
a useful suggestion from one of the referees,
for which the authors are grateful.
%The authors would like to thank the referees
%for this suggestion.
%paper was first written, and would like to thank the referees
%for suggesting that we explain more about the significance of $N_\alpha$.
\end{example}

%For instance, the values
%\[\pi=ENNEENEEE,\quad \tilde{\pi}=\gamma(\pi)=EENEEE,\]
%\[\touch'(\pi)= [3,1,2],\quad \touch'(\tilde{\pi})=[1,2,1,1],\]
%correspond to $r=2,a=3,\alpha=[1,2]$. We can check that
%\[\bounce(NNN\pi)=3,\quad \bounce(NNNN\tilde{\pi})=1.\]
%The proof now follows bijectively by expressing $\chi'(\pi)$ in terms
%of $\chi'(\tilde{\pi})$.

%Since
%\[\chi'(\pi)=
%\begin{cases}
%%%d_+\chi'(\pi) & a=1 \\
%\frac{t^{a-1}}{q-1}[d_-,d_+]d_-^{l(\tilde{\alpha})-l(\alpha)-1}
%\chi'(\tilde{\pi}) & a>1
%\end{cases}
%\]

%We first claim that $\gamma$ restricted to 
%the sub

%\begin{rem}
%In view of Remark \ref{rem.formal}, the compositional shuffle conjecture is reduced to identities expressing $\nabla C_\alpha 1$ for each $\alpha$ as an explicit linear combination of elements of the form $B_{i_1} B_{i_2} \cdots B_{i_k} 1$.
%\end{rem}

\section{Operator relations}
%\section{Dyck path algebra}
%Note that the operators of multiplication by $y_i$ can be expressed
%using 
We have operators %$T_i$, $d_+$ and $d_-$ operators acting on
\begin{equation}
\label{operators}
e_k,d_{\pm},T_i \acts V_*=V_0\oplus V_1\oplus \cdots
\end{equation}
where $e_k$ is the projection onto $V_k$, the others are
defined as above.
It is natural to ask for a complete set of relations between them. 
They are formalized in the following algebra:
\begin{defn}\label{defn:dpa}
The Dyck path algebra $\AA=\AA_q$ (over $R$) is the path algebra of the quiver with vertex set $\Z_{\geq 0}$, arrows $d_+$ from $i$ to $i+1$, arrows $d_-$ from $i+1$ to $i$ for $i\in\Z_{\geq 0}$, and loops $T_1, T_2, \ldots, T_{k-1}$ from $k$ to $k$ subject to the following relations:
$$
(T_i-1)(T_i+q)=0,\quad T_i T_{i+1} T_i = T_{i+1} T_i T_{i+1},
\quad T_i T_j = T_j T_i\quad(|i-j|>1),
$$
$$
T_i d_- = d_- T_i,\quad d_+ T_i = T_{i+1} d_+, \quad T_1 d_+^2 =
d_+^2,\quad d_-^2 T_{k-1} = d_-^2, 
$$
$$
d_-(d_+ d_- - d_- d_+) T_{k-1} = q (d_+ d_- - d_- d_+) d_- \quad(k\geq2),
$$
$$
T_1 (d_+ d_- - d_- d_+) d_+ = q d_+(d_+ d_- - d_- d_+),
$$ 
where in each identity $k$ denotes the index of the vertex where the respective paths begin. 
We have used the same letters $T_i,d_{\pm}$ to label the $i$th loop at every 
node $k$ to match with the previous notation. To distinguish between 
different nodes, we will use $T_i e_k$ where $e_k$ is the idempotent 
associated with node $k$. 
\end{defn}
We will prove
\begin{thm}\label{lem:mainlemma}
The operators \eqref{operators} define
a representation of $\AA$ on $V_*$. Furthermore, we have an isomorphism
of representations
%Let $e_i\in\AA$ denote the idempotent corresponding to the vertex $i$,
%so that in particular 
%if $\AA e_0$ is the left ideal of $\AA$ consisting of paths that begin
%at $0$, then exists a unique linear isomorphism 
\[\varphi:\AA e_0 %=\bigoplus_{k} e_k \AA e_0 
\xrightarrow{\sim} 
%\bigoplus_{k=0}^\infty V_k=
V_*\]
which sends $e_0$ to $1\in V_0$, and maps $e_k \AA e_0$ isomorphically
onto $V_k$.
%respects the direct sum decomposition
%\[\AA e_0 = \bigoplus_{k=0}^\infty e_k \AA e_0.\]
%and commutes with $d_+$, $d_-$ and $T_i$. 
\end{thm}

The proof will occupy the rest of this section.
%We now study the commutation relations between the operators
%where $y_i$ means multiplication by $y_i$. We have
%\begin{prop}\label{prop:relations}
%%For each $k\in\Z_{\geq0}$ we have defined the operators $T_i:V_k\to V_k$ for $1\leq i\leq k-1$, $y_i:V_k\to V_k$ for $1\leq i\leq k$, $d_+:V_k\to V_{k+1}$, $d_-:V_{k+1}\to V_{k}$. They satisfy the following relations:
%$$
%%(T_i-1)(T_i+q)=0,\quad T_i T_{i+1} T_i = T_{i+1} T_i T_{i+1},
%\quad T_i T_j = T_j T_i \quad (|i-j|>1),
%$$
%$$
%y_i y_j = y_j y_i,\quad T_i y_{i+1} T_i = q y_i, \quad T_j y_i = y_i T_j \quad(i\notin\{j,j+1\}),
%$$
%$$
%y_i d_- = d_- y_i,\quad d_+ y_i = T_1 T_2 \cdots T_i y_i (T_1 T_2 \cdots T_i)^{-1} d_+,
%$$
%%$$
%T_i d_- = d_- T_i,\quad d_+ T_i = T_{i+1} d_+, \quad T_1 d_+^2 = d_+^2,\quad d_-^2 T_{k-1} = d_-^2,
%$$
%$$
%d_+ d_- - d_- d_+ = (q-1) T_1 T_2 \cdots T_{k-1} y_k 
%$$
%as operators on $V_k$,
%where $i$ and $j$ can take arbitrary values for which both sides make sense. 
%\end{prop}
We begin by establishing that we have a defined a representation of the algebra.
\begin{lem}
\label{replemma}
The operators $T_i$ and $d_{\pm}$ satisfy the relations of Definition \ref{defn:dpa}.
\end{lem}

\begin{proof}
The first line is just Proposition \ref{Delprop}, and most follow from
definition. The first one that does not is the commutation relation of $d_+$
with $T_i$. We have
\[d_+ T_i(F)=T_1\cdots T_k\left((T_i F)[X+(q-1)y_{k+1}]\right)=\]
%\[T_1\cdots T_kT_i(F)[X+(q-1)y_k]=\]
\[T_1\cdots T_i T_{i+1} T_i \cdots T_k(F)[X+(q-1)y_{k+1}]=\]
\[T_1\cdots T_{i+1} T_{i} T_{i+1} \cdots T_k(F)[X+(q-1)y_{k+1}]=
T_{i+1} d_+(F)\]
using the braiding relations.

For the next, we have
\[d_+^2 F = T_1 T_2\cdots T_{k+1} T_1 T_2 \cdots T_k(F[X+(q-1) y_{k+1}+(q-1) y_{k+2}])=\] 
\[T_2 T_3\cdots T_{k+1} T_1 T_2\cdots T_{k+1} (F[X+(q-1) y_{k+1}+(q-1) y_{k+2}]).\]
The last $T_{k+1}$ can be removed because its argument is symmetric in 
$y_{k+1}$ and $y_{k+2}$, and we obtain $T_1^{-1} d_+^2 F$.

The next identity is more technical.
The operator image of $T_{k-1}-1$ consists of elements of the form $(q y_{k-1} - y_k) F$, where $F$ is symmetric in $y_{k-1}$ and $y_k$. Thus we need to check that $d_-^2$ vanishes on such elements. Let us evaluate $d_-^2$ on 
$$
(q y_{k-1} - y_k) y_{k-1}^a y_{k} ^b F,
$$
where $F$ does not contain the variables $y_{k-1}$ and $y_k$, and $a,b\in\Z_{\geq0}$. We obtain 
$$
(q B_{a+2} B_{b+1} - B_{a+1} B_{b+2}) F. 
$$
This expression is antisymmetric in $a$, $b$ by Corollary 3.4 of \cite{haglund2012compositional}, which implies our identity.

Next using the previous relations and Lemma \ref{lem:commutator} below write

\[d_- (d_+ d_- - d_- d_+) T_{k-1} = (q-1) d_- T_1 T_2 \cdots 
T_{k-1} y_k T_{k-1} = \]
\[q (q-1) d_- T_1 T_2 \cdots T_{k-2} y_{k-1} =\]
\begin{equation} 
q (q-1) T_1 T_2 \cdots T_{k-2} y_{k-1} d_- = q (d_+ d_- - d_- d_+) d_-. 
\label{eq:dminusrel}
\end{equation}
Similarly,
\[T_1 (d_+ d_- - d_- d_+) d_+ = (q-1) T_1 T_1 T_2 \cdots T_{k} y_{k+1}
d_+=\]
\[ (q-1) q^k T_1 y_1 T_1^{-1} T_2^{-1} \cdots T_k^{-1} d_+=
 (q-1) q^k T_1 y_1 T_1^{-1} d_+ T_1^{-1} \cdots T_{k-1}^{-1}=\]
\begin{equation}
\label{eq:dplusrel}
 (q-1) q^k d_+ y_1 T_1^{-1} \cdots T_{k-1}^{-1}
= q d_+(d_+ d_- - d_- d_+).
\end{equation}

\end{proof}

To establish the isomorphism, we first show that we can produce the
operators of multiplication by $y_i$ from $\AA$.
\begin{lem}\label{lem:commutator}
For $F\in V_k$ we have
\begin{equation}
\label{addyeq}
(d_{-} d_{+} - d_{+} d_{-}) F = (q-1) T_1 T_2\cdots T_{k-1} (-y_k
F),\quad
y_i = \frac{1}{q} T_i y_{i+1} T_i.
\end{equation}
\end{lem}
\begin{proof}

%\[y_k=\frac{1}{q-1} T_{k-1}^{-1} \cdots T_1^{-1} (d_+ d_- - d_- d_+),\]
%

First, we endow $V_k$ with the following twisted action of $\Sym[X]$:
\[(F\ast G)[X] = F\left[X+(q-1)\sum_{i=1}^k y_i\right] G,\]
for $F\in \Sym[X]$, and $G\in V_k.$ 
It can be checked that the operators $d_{+}$, $d_{-}$ intertwine this action:
\begin{equation}
\label{twistedeq}
d_{+} (F\ast G) = F\ast d_{+} G, \quad d_{-} (F\ast G) = F\ast d_{-} G
\end{equation}
For the second one, for instance, it suffices to assume that $k=1$.
Then by the definition of $d_-$ given in \eqref{dmdef}, we have
\[d_-(F[X+(q-1)y_1] G)=-F[X] G[X-(q-1)y_1]\pExp[-y_1^{-1}X] \big|_{y_1^{-1}}=F \ast d_-G.\]

We will not need this, but in fact, 
if $\pi_1\in \DD_k$, $\pi_2\in \DD$, and $\pi_1 \cdot \pi_2\in \DD_k$ is their
concatenation, then we must also have that
\[\chi_k(w_1\cdot w_2) = \chi(w_2) \ast \chi_k(w_1). \]
Since the operators on both sides commute with the twisted action of
$\Sym[X]$ introduced above, 
we may assume without loss of generality that $F$ is a polynomial of $y_1,y_2,\ldots,y_k$.

Write the left hand side of the first identity as
\[d_{-} T_1 \cdots T_{k-1} T_k F - 
T_1\cdots T_{k-1} \left((d_{-} F)[X+(q-1) y_k]\right).\]
The operator $d_{-}$ in the first summand involves only the variable $y_{k+1}$. Thus we can write the left hand side as
$$
T_1 \cdots T_{k-1} (d_{-} T_k F - (d_{-}F)[X+(q-1) y_k]).
$$
Hence it is enough to prove 
$$
d_{-} T_k F - (d_{-}F)[X+(q-1) y_k] = (1-q) y_k F.
$$
It is clear that none of the operations involve the variables $y_1, y_2, \ldots, y_{k-1}$. Thus we can assume $F=y_k^i$ for $i\in\Z_{\geq 0}$ without loss of generality. Direct computation gives
$$
T_k(y_k^i)=y_{k+1}^i + (1-q) \sum_{j=1}^i y_k^j y_{k+1}^{i-j}.
$$
Thus the left hand side equals
\[-h_{i+1}[-X] - (1-q)\sum_{j=1}^i y_k^j h_{i-j+1}[-X] + h_{i+1}[-X-(q-1)y_k]=\]
\[- (1-q)\sum_{j=1}^i y_k^j h_{i-j+1}[-X] + (1-q)\sum_{j=1}^{i+1} y_k^j h_{i-j+1}[-X]
= (1-q) y_k^{i+1}.\]

The second relation is easy.
\end{proof}

The operators of multiplication by $y_i$ are characterized by these
relations, and therefore come from elements of $\AA$.
We next establish the relations these operators satisfy within $\AA$:
\begin{lem}\label{lem:addingback}
For $k\in\Z_{>0}$ define elements $y_1,\ldots,y_k\in e_k \AA e_k$ by
solving for $y_iF$ in the identities \eqref{addyeq}.
Then the following identities hold in $\AA$:
$$
y_i T_j = T_j y_i\qquad\text{for $i\notin\{j,j+1\}$,}
$$
$$
y_i d_- = d_- y_i,\qquad d_+ y_i = T_1 T_2 \cdots T_i y_i (T_1 T_2 \cdots T_i)^{-1} d_+,
$$
$$
y_i y_j = y_j y_i\quad\text{for any $i,j$.}
$$
\end{lem}
\begin{proof}
Note that $y_1$ can be written as
$$
y_1 = \frac{1}{q^{k-1}(q-1)} (d_+ d_- - d_- d_+) T_{k-1} \cdots T_1.
$$
Our task becomes easier if we notice that it is enough to check the first identity for $i=1$ and $i=k$, the second one for $i=k$, the third one for $i=1$ and the last one for $i=1$, $j=k$. The other cases can be deduced from these by applying the $T$-operators.

For $j>1$ we have
\[y_1 T_j = \frac{1}{q^{k-1}(q-1)} (d_+ d_- - d_- d_+) T_{k-1} \cdots T_1 T_j=\]
\[\frac{1}{q^{k-1}(q-1)} (d_+ d_- - d_- d_+) T_{j-1} T_{k-1} \cdots T_1
= T_j y_1.\]
Similarly we verify that $y_k$ commutes with $T_j^{-1}$ hence with $T_j$ for $j<k-1$.

Reversing the arguments in (\ref{eq:dminusrel}) and (\ref{eq:dplusrel}) we verify the second and the third identities.

Thus it is left to check that $y_k y_1 = y_1 y_k$. We assume $k\geq 2$. Write the left hand side as
$$
y_k y_1 =\frac{1}{q-1} T_{k-1}^{-1} \cdots T_1^{-1} (d_+ d_- - d_- d_+) y_1
= \frac{1}{q-1} T_{k-1}^{-1} \cdots T_1^{-1} (T_1 y_1 T_1^{-1}) (d_+ d_- - d_- d_+)
$$
using the already established commutation relations and that $k\geq 2$ to swap $T_1 y_1 T_1^{-1}$ and $d_-$. Performing the cancellation we obtain $y_1 y_k$.
\end{proof}

The following lemma completes the proof of the theorem:
\begin{lem}
\label{lem:algbasis}
The elements of the form 
\begin{equation}
\label{basiseq}
d_-^m y_1^{a_1} \cdots y_{k+m}^{a_{k+m}} d_+^{k+m} e_0 
\end{equation}
with $a_{k+1}\geq a_{k+2}\geq\cdots \geq a_{k+m}$ form a basis of $\AA e_0$. 
Furthermore, the representation $\varphi$ maps these elements to a basis of $V_*$.
\end{lem}
\begin{proof}
We first show that elements of the form \eqref{basiseq}, with no
condition
on the $a_i$ span $\AA$.
It suffices to check that the span of these elements is 
invariant under $d_-$, $T_i$ and $d_+$. This can be done by applying
the following reduction rules that follow 
from the definition of $\AA$ and Lemma \ref{lem:addingback}:
\[T_i d_- \to d_- T_i,\quad T_j y_i \to y_i T_j \quad(i\notin \{j,j+1\}),\]
\[T_i y_i \to y_{i+1} T_i + (1-q) y_i, \quad T_i y_{i+1} \to y_{i} T_i + (q-1) y_i,\]
\[T_i d_+^{k+m} e_0 \to d_+^{k+m} e_0,\]
\[d_+ d_- \to d_- d_+ + (q-1) T_1 T_2 \cdots T_{k-1} y_k,\quad y_i d_- \to d_- y_i. \]

The next step is to reduce the spanning set:
We can use the following identity, which follows from $d_-^2 T_{k-1}=d_-^2$:
\[d_-^m (1-T_j) y_1^{a_1}\cdots y_{k+m}^{a_{k+m}} d_+^{k+m} e_0 = 0 \qquad (k<j<k+m).\]
Note that $T_j$ commutes with $y_j y_{j+1}$. Suppose $a_j<a_{j+1}$. Then we can rewrite the above identity as
\[0=d_-^m y_1^{a_1}\cdots y_j^{a_j} y_{j+1}^{a_j} 
(1-T_j) y_{j+1}^{a_{j+1}-a_j} y_{j+2}^{a_{j+2}} \cdots y_{k+m}^{a_{k+m}} d_+^{k+m} e_0.\]
Using $T_j y_{j+1} = y_j (T_j+(q-1))$, $T_j y_r = y_r T_j$ for
$r>j+1$, and $T_j d_+^{k+m} e_0=d_+^{k+m} e_0$ we can rewrite the
identity as vanishing of a linear combination of terms of the form
\eqref{basiseq}, 
and the lexicographically smallest term is precisely 
\[d_-^m y_1^{a_1}\cdots y_{k+m}^{a_{k+m}} d_+^{k+m} e_0.\]
Thus we can always reduce terms of the form \eqref{basiseq} which
violate the condition $a_{k+1}\geq a_{k+2}\geq\cdots \geq a_{k+m}$ to
a linear combination of lexicographically greater terms, showing that
the subspace in the lemma at least spans $\AA e_0$.

We now show that they map to a basis of $V_*$, which also establishes
that they are independent, completing the proof.
Consider the image of the 
elements of our spanning set
\[d_-^m y_1^{a_1} \cdots y_{k+m}^{a_{k+m}} d_+^{k+m} (1) 
= d_-^m (y_1^{a_1} \cdots y_{k+m}^{a_{k+m}}) \]
\begin{equation}
\label{eq:elements}
(-1)^m y_1^{a_1} y_2^{a_2} \cdots y_k^{a_k} B_{a_{k+1}+1}
B_{a_{k+2}+1} \cdots B_{a_{k+m}+1} (1). 
\end{equation}
Notice that $\lambda:=(a_{k+1}+1,a_{k+2}+1,\ldots,a_{k+m}+1)$ is a partition, so 
\[B_{a_{k+1}+1} B_{a_{k+2}+1} \cdots B_{a_{k+m}+1} (1) \]
is a multiple of the Hall-Littlewood polynomial
$\Ht_{\lambda'}[-X;1/q,0]$. These polynomials form a basis of the space of
symmetric functions, thus the elements \eqref{eq:elements} form a
basis of $\bigoplus_{k\in\Z_{\geq0}} V_k$.

\end{proof}

\section{Conjugate structure}

It is natural to ask if there is a way to extend 
$\nabla$ to the spaces $V_k$, recovering the original operator at $k=0$. 
What we have found is that it is simpler to extend the composition 
\begin{equation}
\label{omega1}
\mathcal{N}(F)=\nabla \bar\omega F=\nabla \omega\overline{F}
\end{equation}
where the conjugation simply makes the substitution 
$(q,t)=(q^{-1},t^{-1})$, $\omega(F)=F[-X]$ is the Weyl involution
up to a sign, and $\bar\omega$ denotes the composition.
This is a very interesting operator, which in fact is an antilinear 
involution on $\Sym[X]$ corresponding to dualizing vector bundles
and tensoring with $\mathcal{O}(1)$ 
in the Haiman-Bridgeland-King-Reid picture, which identifies $\Sym[X]$
with the equivariant $K$-theory of the Hilbert scheme of points in the 
complex plane \cite{BKR}. The key to our proof is to extend this operator 
to an antilinear involution on every $V_k$, suggesting that $V_k$
should have a geometric interpretation as well. 
Since this paper was completed, E. Gorsky and the authors have discovered
this connection in terms of a smooth subscheme of the flag Hilbert scheme. 
This will be explained in an upcoming paper,
which also explains a family of ways to extend the operator 
$\nabla$ itself in addition to the involution $\mathcal{N}$,
also in terms of line bundles.

We will define the operator, which was discovered experimentally to
have nice properties, by explicitly constructing the action of
$\AA$ conjugated by the conjectural involution $\mathcal{N}$.
Let $\AA^*=\AA_{q^{-1}}$, and label the corresponding generators by 
$d^*_{\pm},T_i^*,e_i^*$. Denote by $z_i$ the image of $y_i$ under the 
isomorphism from $\AA$ to $\AA^*$ which sends generators to 
generators, and is antilinear with respect to $q\mapsto q^{-1}$. 
\begin{thm}
\label{Astarthm}
There is an action of $\AA^*$ on $V_*$ given by the assignment 
\begin{equation}
\label{eq:dplusstar}
T_i^*=T^{-1}_i,\quad d^*_{-}=d_-,\quad e_i^*=e_i,\quad 
%\begin{equation}%
(d_+^* F)[X] = \gamma F[X+(q-1)y_{k+1}],
\end{equation}
% \qquad (k\in\Z_{\geq0},F\in V_k),
where $F\in V_k$ and $\gamma$ is the operator which sends $y_i$ to $y_{i+1}$ for $i=1,\ldots,k$ and $y_{k+1}$ to $t y_1$. 
Furthermore, it satisfies the additional relations 
\begin{equation}
\label{addreleq}
z_{i+1} d_+ = d_+ z_i,\quad y_{i+1} d_+^* = d_+^* y_i,\quad z_1 d_+ = -y_1 d_+^* t q^{k+1},\quad 
d_+^* d_+^m (1) = d_+^{m+1}(1) 
\end{equation}
for any $m\geq 0$. 
\end{thm}

The statement is equivalent to validity of a certain set of relations 
satisfied by the operators. These will be verified in the following propositions.

First we list the obvious relations:
\begin{prop}
$$
d_+^* T_i^{-1} = T_{i+1}^{-1} d_+^*, \qquad T_1^{-1} d_+^{*2} = d_+^{*2}, 
\qquad d_+^* y_i = y_{i+1} d_+^*.
$$
\end{prop}
\begin{proof}
Easy from the definition.
\end{proof}

To verify the rest of the relations, we use an approach similar to the one used in the proof 
of Lemma \ref{lem:commutator}. 
Now we need not just one, but a family of twisted multiplications:
For $F\in\Sym[X]$, $G\in V_k$, $m=0,1,2,\ldots,k$ put 
\[(F \ast_m G)[X] = F\left[X + (q-1)\left(\sum_{i=1}^m t y_i + \sum_{i=m+1}^k y_i\right)\right] G. \]
It is not hard to show that they satisfy 
\begin{equation}\label{eq:twisted mult n}
d_+^*(F\ast_m G) = F \ast_{m+1} d_+^* G,\quad 
d_-(F\ast_m G) = F\ast_m d_- G,
\end{equation}
where $F\in \Sym[X]$, and $G\in V_k$,
the first identity holds for $0\leq m\leq k$, and the second one for $0\leq m<k$.

Let us first verify
\begin{prop}
\begin{equation}\label{eq:conjrel1}
d_-(d_+^* d_- - d_- d_+^*) T_{k-1}^{-1} = q^{-1} (d_+^* d_- - d_- d_+^*) d_-\quad(k\geq 2).
\end{equation}
\end{prop}
\begin{proof}
Rewrite it as 
$$
d_+^* d_-^2 - d_- d_+^* d_- (T_{k-1} + q) + q d_-^2 d_+^* = 0. 
$$
Multiplying both sides by $q-1=T_{k-1}-1 + q-T_{k-1}$ produces an equivalent relation, which can be reduced to 
$$
(d_+^* d_-^2 - (q+1) d_- d_+^* d_- + q d_-^2 d_+^*)(T_{k-1}+q) = 0. 
$$
Note that the image of $T_{k-1}+q$ consists of elements which are symmetric in $y_{k-1}$, $y_k$. Let 
$$
A=d_+^* d_-^2 - (q+1) d_- d_+^* d_- + q d_-^2 d_+^*. 
$$
It is enough to show that $A$ vanishes on elements of $V_k$ that are symmetric in $y_{k-1}$, $y_k$. We have (recall that $k\leq 2$) 
$$
A(F\ast G) = F \ast_1 A (G),\quad A y_i  = y_{i+1} A \quad (F\in\Sym[X],\, G\in V_k,\, i<k-1),
$$
Thus it is enough to verify vanishing of $A$ on symmetric polynomials of $y_{k-1}, y_k$. We evaluate $A$ on $y_{k-1}^a y_k^b$:
\[
A (y_{k-1}^a y_k^b) = (\Gamma_+(t(q-1)y_1) B_{a+1} B_{b+1} - (q+1) B_{a+1} \Gamma_+(t(q-1)y_1) B_{b+1} 
\]
\[
+ q B_{a+1} B_{b+1} \Gamma_+(t(q-1)y_1) ) 1,
\]
where $\Gamma_+(Z)$ is the operator $F[X]\to F[X+Z]$. For any monomial $u$ and integer $i$ we have operator identities 
$$
\Gamma_+(u) B_i = (B_i - u B_{i-1}) \Gamma_+(u),\quad 
B_i \Gamma_+(-u) = \Gamma_+(-u)(B_i - u B_{i-1}),
$$
thus we have 
\begin{align*}
\Gamma_+(t(q-1)y_1) B_{a+1} B_{b+1} &= \Gamma_+(-ty_1) (B_{a+1} - qty_1 B_a) (B_{a+1} - qty_1 B_a) \Gamma_+(qty_1), \\
B_{a+1} \Gamma_+(t(q-1)y_1) B_{b+1} &= \Gamma_+(-ty_1) (B_{a+1} - ty_1 B_a) (B_{a+1} - qty_1 B_a) \Gamma_+(qty_1),\\
B_{a+1} B_{b+1} \Gamma_+(t(q-1)y_1) &= \Gamma_+(-ty_1) (B_{a+1} - ty_1 B_a) (B_{a+1} - ty_1 B_a) \Gamma_+(qty_1). 
\end{align*}
Performing the cancellations we arrive at 
$$
A (y_{k-1}^a y_k^b) = \Gamma_+(-ty_1) (t y_1 (1-q) (B_{a} B_{b+1} - q B_{a+1} B_b)) 1. 
$$
This expression is antisymmetric in $a$, $b$ by Corollary 3.4, \cite{haglund2012compositional}. Thus \eqref{eq:conjrel1} is true. 
\end{proof}

Next we have to check that 
\begin{prop}
\[
T_1^{-1} (d_+^* d_- - d_- d_+^*) d_+^* = q^{-1} d_+^*(d_+^* d_- - d_- d_+^*). 
\]
\end{prop}
\begin{proof}
Multiplying both sides by $q T_1$ and using the easier relations, rewrite it as
$$d_+^{*2} d_- - (T_1+q) d_+^* d_- d_+^* + q d_- d_+^{*2} = 0.$$
Again, because of the commutation relations with the twisted multiplication by symmetric functions and $y_i$, it is enough to evaluate the left hand side on $y_k^a$ for all $a\in\Z_{\geq0}$. We obtain 
$$
- h_{a+1}[-X - t(q-1)(y_1+y_2)] + (T_1+q) h_{a+1}[-X-t(q-1)y_1] - q h_{a+1}[-X]. 
$$

We use the identity $h_n[X+Y]=\sum_{i+j=n} h_i[X] h_j[Y]$ to write the left hand side as a linear combination of terms $h_{a+1-b}[-X]$ with $b>0$. The coefficient in front of each term with $b>0$ is 
$$
- h_{b}[t(1-q)(y_1+y_2)] + (T_1+q) h_{b}[t(1-q)y_1]. 
$$ 
By a direct computation:
\[(T_1+q) h_{b}[t(1-q)y_1] = (T_1+q) (1-q) t^b y_1^b =\]
\[(1-q) t^b  (y_1^b + (1-q)\sum_{i=1}^{b-1} y_1^i y_2^{b-i} + y_2^b) = h_{b}[t(1-q)(y_1+y_2)],\]
and we are done. 
\end{proof}

At this point, we have established the fact that the operators given by \eqref{eq:dplusstar} define an action of $\AA^*$ on $V_*$.  Also we have established the second relation in \eqref{addreleq}. The last relation is obvious. The first and the third are verified below:

\begin{prop}
$$
z_1 d_+ = -y_1 d_+^* t q^{k+1},\qquad z_{i+1} d_+ = d_+ z_i.
$$
\end{prop}
\begin{proof}
Using \eqref{eq:twisted mult n} we see that the operator $y_1 d_+^*$ satisfies the following two properties:
\[
y_1 d_+^* y_i = y_{i+1} y_1 d_+^*,\qquad y_1 d_+^* (F\ast G) = F \ast_1 y_1 d_+^*(G)
\]
for $F\in \Sym[X]$, $G\in V_k$, $i=1,\ldots,k$.
By definition (on $V_{k}$)
$$
z_1 = \frac{q^{k-1}}{q^{-1}-1} (d_+^* d_- - d_- d_+^*) T_{k-1}^{-1} \cdots T_1^{-1},
$$
thus (again, on $V_k$)
$$
z_1 d_+ = \frac{q^{k}}{q^{-1}-1} (d_+^* d_- - d_- d_+^*) T_{k}^{-1} \cdots T_1^{-1} d_+.
$$
From this expression we see, using \eqref{eq:twisted mult n} once again, that $z_1 d_+$ satisfies the same two properties as the operator $y_1 d_+^*$:
$$
z_1 d_+ y_i = y_{i+1} z_1 d_+,\qquad z_1 d_+ (F\ast G) = F \ast_1 z_1 d_+(G)
$$
for $F\in \Sym[X]$, $G\in V_k$, $i=1,\ldots,k$. Thus it is enough to verify the first identity on $1\in V_k$. The right hand side is $-t q^{k+1} y_1$. The left hand side is
$$
\frac{q^k}{q^{-1}-1}(d_+^*-1)d_-(1) = \frac{q^k}{q^{-1}-1}(X+t (q-1)y_1-X) = -t q^{k+1} y_1,
$$
so the first identity holds.

It is enough to verify the second identity for $i=1$ because the general case can be deduced follows from this one by applying the $T$-operators. For the identity $z_2 d_+ = d_+ z_1$, expressing $z_1$, $z_2$ in terms of $d_-$, $d_+^*$ and the $T$-operators, we arrive at the following equivalent identity:
$$
T_1^{-1} d_+(d_+^* d_- - d_- d_+^*) = (d_+^* d_- - d_- d_+^*) d_+.
$$
If we denote by $A$ either of the two sides, we have
$$
A (F\ast G) = F\ast_1 A(G),\qquad A y_i = T_2 T_3\cdots T_{i+1} y_{i+1} (T_2 T_3 \cdots T_{i+1})^{-1} A
$$
for $F\in \Sym[X]$, $G\in V_k$, $i=1,\ldots,k-1$. Thus it is enough to verify the identity on $y_k^a\in V_k$ ($a\in\Z_{\geq0}$). Applying $T_k^{-1} T_{k-1}^{-1} \cdots T_2^{-1}$ to both sides, the identity to be verified is
$$
T_k^{-1} T_{k-1}^{-1} \cdots T_1^{-1} d_+ (d_+^* d_- - d_- d_+^*)(y_k^a)
=
(d_+^* d_- - d_- d_+^*) T_{k-1}^{-1} \cdots T_1^{-1} d_+(y_k^a). 
$$
The left hand side is evaluated to 
$$
-h_{a+1}[-X-t(q-1)y_1-(q-1)y_{k+1}] + h_{a+1}[-X-(q-1)y_{k+1}].
$$
The right hand side is evaluated to 
$$
(d_+^* d_- - d_- d_+^*) T_k(y_{k+1}^a) = F[X+t(q-1)y_1] - F[X]
$$
with
\[
F[X] = -h_{a+1}[-X] - (1-q)\sum_{i=0}^{a-1} y_{k+1}^{a-i} h_{i+1}[-X]
\]
\[
= -h_{a+1}[-X+(1-q)y_{k+1}] + (1-q)y_{k+1}^{a+1},
\]
and the identity follows.
\end{proof}
This completes our proof of Theorem \ref{Astarthm}.

We also have the following Proposition, which we will use to connect the conjugate
action with $N_{\alpha}$.
\begin{prop}\label{prop:recy}
For a composition $\alpha$ of length $k$ let 
$$
y_\alpha = y_1^{\alpha_1-1}\cdots y_k^{\alpha_k-1} \in V_k. 
$$
Then the following recursions hold:
$$
y_{1\alpha} = d_+^* y_\alpha,\quad y_{a \alpha} = \frac{t^{1-a}}{q-1}(d_+^* d_- - d_- d_+^*) \sum_{\beta\models a-1} q^{1-l(\beta)} d_-^{l(\beta)-1}(y_{\alpha\beta})\quad(a>1). 
$$
\end{prop}
\begin{proof}
The first identity easily follows from the explicit formula for $d_+^*$. For $i=1,2,\ldots,k-1$ we have 
$$
(d_- d_+^* - d_+^* d_-) y_i = y_{i+1} (d_- d_+^* - d_+^* d_-). 
$$
Therefore it is enough to verify the following identity for any $a\in\Z_{\geq1}$:
\begin{equation}\label{eq:y1a}
(q-1) t^a y_1^a = (d_+^* d_- - d_- d_+^*) \sum_{\beta\models a}
q^{1-l(\beta)} d_-^{l(\beta)-1}(y_k^{\beta_1-1} \cdots y_{k+l(\beta)-1}^{\beta_{l(\beta)}-1}) \in V_k. 
\end{equation}
We group the terms on the right hand side by $b=\beta_1-1$ and the sum becomes 
\[
\sum_{b=0}^{a-1} y_k^b \sum_{\beta\models a-b-1} q^{-l(\beta)} d_-^{l(\beta)}\left(y_{k+1}^{\beta_1-1}\cdots y_{k+l(\beta)}^{\beta_{l(\beta)}-1}\right) 
\]
\[
=\sum_{b=0}^{a-1} y_k^b \sum_{\beta\models a-b-1} q^{-l(\beta)} (-1)^{l(\beta)} B_{\beta_1} \cdots B_{\beta_{l(\beta)}}(1) 
=\sum_{b=0}^{a-1} y_k^b h_{a-b-1}[q^{-1}X]. 
\]
We have used the identity 
\begin{equation}\label{eq:hnBn}
h_n[q^{-1}X] = \sum_{\alpha\models n} q^{-l(\alpha)} (-1)^{l(\alpha)} B_\alpha(1),
\end{equation}
which can be obtained by applying $\bar\omega$ to Proposition 5.2 of \cite{haglund2012compositional}:
$$
h_n[-X] = \sum_{\alpha\models n} C_\alpha(1). 
$$
Thus the right hand side of (\ref{eq:y1a}) is evaluated to the following expression:
\[
(d_+^* d_- - d_- d_+^*) \sum_{b=0}^{a-1} y_k^b q^{-(a-b-1)} h_{a-b-1}[X]
\]
\[
= -\sum_{b=0}^{a-1}\left(\Gamma_{+}(t(q-1)y_1) B_{b+1} - B_{b+1} \Gamma_{+}(t(q-1)y_1)\right) h_{a-b-1}[q^{-1} X]. 
\]
\[
= - \sum_{b=0}^{a-1} \Gamma_+(-t y_1) \left((B_{b+1} - qt y_1 B_{b}) - (B_{b+1} - t y_1 B_{b})\right) (h_{a-b-1}[q^{-1}X + ty_1]) 
\]
\[
= (q-1) t y_1 \Gamma_+(-t y_1) \sum_{b=0}^{a-1} B_{b} ( h_{a-b-1}[q^{-1}X+t y_1]). 
\]
Thus we need to prove 
$$
\sum_{b=0}^{a-1} B_b (h_{a-b-1}[q^{-1} X+ty_1]) = t^{a-1} y_1^{a-1}. 
$$
Then the left hand side as a polynomial in $y_1$ indeed has the right coefficient of $y_1^{a-1}$. The coefficient of $y_1^i$ for $i<a-1$ is 
$$
t^i \sum_{b=0}^{a-1-i} B_b(h_{a-b-1-i}[q^{-1} X]). 
$$
So it is enough to show:
$$
\sum_{b=0}^m B_b(h_{m-b}[q^{-1}X]) = 0 \quad (m\in\Z_{>0}). 
$$
Using (\ref{eq:hnBn}) again we see that the left hand side equals 
$$
B_0(h_m[q^{-1}X]) - q h_m[q^{-1}X] = (B_0-q)(-q^{-1} C_m(1)) = 0 
$$
because $B_0 C_m = q C_m B_0$ by Proposition 3.5 of \cite{haglund2012compositional} and $B_0(1)=1$. 
\end{proof}

\section{The Involution}

\begin{defn}
Consider $\AA$ and $\AA^*$ as algebras over $\Q(q,t)$, and let 
$\tilde{\AA}=\tilde{\AA}_{q,t}$ be the quotient of the free product of 
$\AA$ and $\AA^*$ by the relations 
\[d_-^* = d_-,\quad T_i^*=T_i^{-1},\quad e_i^*=e_i,\]
\[z_{i+1} d_+ = d_+ z_i,\quad y_{i+1} d_+^* = d_+^* y_i,\quad z_1 d_+ = -y_1 d_+^* t q^{k+1},\]
\end{defn}
\begin{rem}
For any $k\geq 0$ the \emph{affine Hecke algebra} $\mathrm{AHA}_k$ is the algebra generated over $\Q(q)$ by $T_1,\ldots,T_{k-1}, y_1^{\pm1}, \ldots, y_k^{\pm1}$ modulo relations
\[
(T_i-1)(T_i+q)=0,\quad T_i T_{i+1} T_i = T_{i+1} T_i T_{i+1},
\quad T_i T_j = T_j T_i\quad(|i-j|>1),
\]
\[
y_i T_j = T_j y_i\quad (i\notin\{j,j+1\}),
\quad y_i y_j = y_j y_i,
\quad
T_i y_{i+1} T_i = q y_i.
\]
The \emph{positive part} $\mathrm{AHA}_k^+$ is defined as the subalgebra of $\mathrm{AHA}_k$ generated by $T_i$ and $y_i$, or equivalently as the algebra generated over $\Q(q)$ by $T_1,\ldots,T_{k-1}, y_1, \ldots, y_k$ modulo the same relations. We have a natural homomorphism $\mathrm{AHA}_k^+\to e_k\AA e_k$ which can be shown to be injective using Lemma \ref{lem:algbasis}. It is tempting to guess that in a similar way the subalgebra of $e_k\tilde\AA e_k$ generated by $T_i$, $y_i$ and $z_i$ is isomorphic to the positive part $\mathrm{DAHA}_k^{++}$ of the \emph{double affine Hecke algebra} $\mathrm{DAHA}_k$. To fix a version of $\mathrm{DAHA}_k^{++}$ which is close to our notations we start with relations (2.1-2.7) in \cite{shiff2011elliptic} and perform substitutions  $q=t^{-1}$, $v=q^{\frac12}$, $T_i=q^{\frac12} T_i^{-1}$, $X_i=y_i$, $Y_i=z_i$ followed by reversal of the order of generators in each monomial. So $\mathrm{DAHA}_k^{++}$ is defined over $\Q(q,t)$ by generators $T_1,\ldots,T_{k-1}, y_1, \ldots, y_k, z_1, \ldots, z_k$ and relations of two copies of $\mathrm{AHA}_k$ (the second copy is transformed by $T_i\to T_i^{-1}$, $q\to q^{-1}$)
\[
(T_i-1)(T_i+q)=0,\quad T_i T_{i+1} T_i = T_{i+1} T_i T_{i+1},
\quad T_i T_j = T_j T_i\quad(|i-j|>1),
\]
\[
y_i T_j = T_j y_i\quad (i\notin\{j,j+1\}),
\quad y_i y_j = y_j y_i,
\quad
T_i y_{i+1} T_i = q y_i,
\]
\[
z_i T_j = T_j z_i\quad (i\notin\{j,j+1\}),
\quad z_i z_j = z_j z_i,
\quad
T_i z_{i} T_i = q z_{i+1},
\]
and two extra relations. The first one is
\[
z_2 y_1 = q y_1 T_1^{-2} z_2\quad\Leftrightarrow\quad q y_2 z_1 = z_1 T_1^2 y_2,
\]
which \emph{can} be deduced in $\tilde{A}$ from \eqref{addyeq} and $z_2 d_+ = d_+ z_1$. The second relation is
\[
z_1 y_1 \cdots y_k = t y_1 \cdots y_k z_1.
\]
The following identity can be deduced from the rest of the relations:
\[
y_2 \cdots y_k z_1 = q^{1-k} z_1 T_1 \cdots T_{k-1} T_{k-1} \cdots T_1 y_2\cdots y_k.
\]
Thus we expect to have
\[
z_1 y_1 = q^{1-k}\, t\, y_1 z_1 T_1 \cdots T_{k-1} T_{k-1} \cdots T_1.
\]
However this \emph{does not} hold in $\tilde{A}$. Instead we have
\[
z_1 y_1 = q^{1-k}\, t\, y_1 z_1 T_1 \cdots T_{k-1} T_{k-1} \cdots T_1 \;+\; q t  y_1  d_- d_+^* T_{k-1} \cdots T_1.
\]
So we see that we \emph{do not} obtain a natural homomorphism $\mathrm{DAHA}_k^{++}\to e_k \tilde\AA e_k$. One way to repair the situation is to introduce the ``partially symmetrized'' $\mathrm{SDAHA}_{k,\infty}^{++}$ by starting with the DAHA in infinitely many generators $T_i$, $z_i$, $y_i$ ($i=1,2,3,\ldots$), and then symmetrizing in generators with $i>k$. For instance, for $k=0$ we expect $e_0 \tilde \AA e_0$ to coincide with the positive part of the elliptic Hall algebra, which is the stable limit of spherical DAHAs as shown in \cite{shiff2011elliptic}. Details of this construction will be provided in a future publication.
\end{rem}
We now prove 
\begin{thm}
\label{Atildethm}
The operations $T_i$, $d_-$, $d_+$, $d_+^*$, $e_i$ define an action of $\tilde{\AA}$ on 
$V_*$. Furthermore, the kernel of the natural map $\tilde{\AA} e_0 \rightarrow V_*$
that sends $f e_0$ to $f(1)$ is given by $I e_0$ where $I \subset 
\tilde{\AA}$ is the ideal generated by 
\begin{equation}
\label{defI}
I=\langle d_+^* d_+^m-d_+^{m+1}|\quad m \geq 0\rangle. 
\end{equation}
In particular, we have an isomorphism $V_* \cong \tilde{\AA} e_0/Ie_0$. 
\end{thm}

\begin{proof}
Theorem \ref{Astarthm} shows that we have a map of modules
$\tilde{\AA} e_0 \rightarrow V_*$, that restricts to the isomorphism 
of Theorem \ref{lem:mainlemma} on the subspace $\AA e_0$, so in
particular is surjective. Furthermore, the last relation
of \eqref{addreleq} shows that it descends to a map
$\tilde{\AA} e_0/Ie_0 \mapsto V_*$, which must still be surjective. We have the following commutative diagram:

\[
\begin{tikzcd}
\tilde{\AA} e_0/Ie_0 \arrow{r} & V_*\\
\AA e_0 \arrow{u} \arrow{ur}{\sim} & 
\end{tikzcd}
\]

Thus we have an inclusion $\AA e_0 \subset \tilde{\AA}e_0/Ie_0$ and it remains to show that
the image of $\AA e_0$ in $\tilde{\AA}e_0/Ie_0$ is the entire space.
We do so by induction: notice that both $\AA e_0$ and $\tilde{\AA}e_0/Ie_0$ have a grading by the total degree in $d_+$, $d_+^*$ and $d_-$, as all the relations are homogeneous. 
For instance, $y_i$ and $z_i$ have degree $2$, and $T_i$ has degree $0$ for all
$i$. Denote the space of elements of degree $m$ in $\AA e_0$, $\tilde{\AA}e_0/Ie_0$ by $V^{(m)}$, $W^{(m)}$ respectively. We need to prove $V^{(m)}=W^{(m)}$. The base cases $m=0$, $m=1$ are clear.

For the induction step, suppose $m>0$, $V^{(i)}=W^{(i)}$ for $i\leq m$ and let $F\in V^{(m)}$. It is enough to show that $d_+^* F \in V^{(m+1)}$.
By Lemma \ref{lem:algbasis}, we can assume that
$F$ is in the canonical form \eqref{basiseq}. 
We therefore must check three cases: $F=d_+^m (1)$ for $1\in V_0$,
$F=y_i G$ for $G\in V^{(m-2)}$, and $F=d_- (G)$ for $G\in V^{(m-1)}$. 
In the first case we have $d_+^* F = d_+^{m+1} 1$. 
In the second case we have $d_+^* (F) = y_{i+1} d_+^* (G)$. 

In the third case, we have 
\[d_+^* F = d_+^* d_- G = d_- d_+^* G + (q^{-1} - 1) T_1^{-1}\cdots T_{k-1}^{-1} z_k G.\] 
Now we use expansion of $G$ in terms of the generators $T_i$, $d_+$ and $d_-$. Because of the commutation 
relations between $T_i$ and $z_j$ it is enough to consider two cases: $G=d_+ G'$ and $G=d_- G'$ for $G'\in V^{(m-2)}$. In the first case we have $z_k G = d_+ z_{k-1} G'$ if $k>1$ and $z_1 G = -y_1 d_+^* t G'$ if $k=1$. In the second case we have $z_k G = d_- z_k G'$. In all cases the claim is reduced to the induction hypothesis.
\end{proof}

Now by looking at the defining relations of $\tilde{\AA}$, we make 
the remarkable observation that there exists an
involution $\iota$ of $\tilde{\AA}$ that permutes $\AA$ and $\AA^*$
and is antilinear with respect to
the conjugation $(q,t)\mapsto (q^{-1},t^{-1})$ on the ground field
$\Q(q,t)$! Furthermore, this involution preserves the ideal $I$,
and therefore induces an involution on $V_*$ via the isomorphism of Theorem
\ref{Atildethm}.

\begin{thm}
\label{mainthm}
There exists a unique antilinear degree-preserving automorphism $\mathcal{N} : V_*
\rightarrow V_*$ satisfying
\[\mathcal{N}(1)=1,\quad \mathcal{N} T_i = T_i^{-1} \mathcal{N},
\quad \mathcal{N} d_- = d_- \mathcal{N},\quad \mathcal{N} d_+ = d_+^*
\mathcal{N},
\quad \mathcal{N} y_i = z_i \mathcal{N}.\]
Moreover, we have
\begin{enumerate}[label=(\roman*)]%[a)]%[label=(\roman*),ref=(\roman*)]
\item \label{invpart} $\mathcal{N}$ is an involution, i.e. $\mathcal{N}^2=\Id$.
%\[\mathcal{N}^2=1,\quad \mathcal{N} f(q,t) F=f(q^{-1},t^{-1})
%\mathcal{N} F\]
\item \label{Nalphapart} For any composition $\alpha$ we have
$$
\mathcal{N}(y_\alpha) = q^{\sum(\alpha_i-1)} N_\alpha.
$$
\item \label{nabpart} On $V_0=\Sym[X]$, we have $\mathcal{N} = \nabla
  \bar\omega$, where $\bar\omega$ is the involution sending $q$, $t$, $X$ to $q^{-1}$, $t^{-1}$, $-X$ resp. (see \eqref{omega1}).
\end{enumerate}
\end{thm}
\begin{proof}

The automorphism is induced from the involution of $\tilde{\AA}$, from which
part \ref{invpart} follows immediately. Part \ref{Nalphapart} follows
from applying $\mathcal{N}$ to the relations of Proposition \ref{prop:recy}.

Finally, let $D_1, D_1^*: V_0\to V_0$ be the operators
\[(D_1 F)[X] = F[X+(1-q)(1-t) u^{-1}] \pExp[-uX] |_{u^1},\]
\[(D_1^* F)[X] = F[X-(1-q^{-1})(1-t^{-1}) u^{-1}] \pExp[uX] |_{u^1},\]
and let $\underline{e}_1: V_0\to V_0$ be the operator of multiplication by $e_1[X]=X$.
It is easy to verify that
\[D_1 = -d_- d_+^*,\quad \underline{e}_1 = d_- d_+,\quad \bar\omega D_1 = D_1^*\bar\omega.\]
Thus it follows that
\[\mathcal{N} D_1 = -\underline{e}_1 \mathcal{N},\quad \mathcal{N} \underline{e}_1 = -D_1 \mathcal{N}.\]
Let $\nabla'=\mathcal{N}\bar\omega$ on $V_0$.
Then 
\[\nabla'(1) = 1,\quad \nabla' \underline{e}_1 = 
D_1 \nabla',\quad \nabla' D_1^* = -\underline{e}_1 \nabla'.\]
It was shown in \cite{garsia1998explicit} that $\nabla$ satisfies the
same commutation relations, 
and that one can obtain all symmetric functions starting from $1$ and 
successively applying $\underline{e}_1$ and $D_1^*$. Thus
$\nabla=\nabla'$, 
proving part \ref{nabpart}.
\end{proof}

% We first need to generalize the twisted multiplication because $d_+^*$ does not commute with the twisted multiplication introduced before. 
%
%
%Our main discovery is the following non-trivial action of $\AA^*$ on $V_\cdot$:
%
%In fact $\AA^*$ appears to be isomorphic to $\AA$ in the following way:
%\begin{prop}
%There is an algebra isomorphism $\AA^*\to\AA$ given on the generators as
%$$
%d_-^* \to d_-^*,\quad T_i^* \to T_i, \quad d_+^* \to d_+ q^{-k}.
%$$
%\end{prop} 

The compositional shuffle conjecture now follows easily:
\begin{thm}
\label{shuffthm}
For a composition $\alpha$ of length $k$, we have
\[\nabla C_{\alpha_1}\cdots C_{\alpha_k}(1)= D_{\alpha}(X;q,t).\]
\end{thm}
\begin{proof}
Using Theorems \ref{thm:recN} and \ref{mainthm}, we have
\[D_{\alpha}(q,t)=d_-^{k} (N_\alpha) = d_-^k(\mathcal{N}(q^{|\alpha|-k} y_\alpha))=\mathcal{N}(q^{|\alpha|-k} d_-^k(y_\alpha))=\]
\[\mathcal{N}\left(q^{|\alpha|-k} (-1)^{k} B_{\alpha}(1)\right) = 
\mathcal{N}\bar\omega C_{\alpha}(1)=\nabla
C_{\alpha}(1).\]
\end{proof}

%\printbibliography

\bibliographystyle{amsalpha}
\bibliography{refs}

\end{document}